\documentclass[reqno]{amsart}

\usepackage{graphicx}
\usepackage{amssymb}
\usepackage{enumitem}
\usepackage{mathtools}
\usepackage{xcolor}
\usepackage{soul}
\usepackage{comment}
\usepackage{tikz}
\usepackage{pgfplots} 
\usetikzlibrary{arrows.meta, decorations.pathreplacing, calc, shapes}
\usepackage{bbm}
\usepackage[colorlinks=true, linkcolor=blue, hypertexnames=false]{hyperref} 

\newtheorem{theorem}{Theorem}[section]

\newtheorem{lemma}[theorem]{Lemma}

\newtheorem{coro}{Corollary}[section]
\newtheorem{propo}{Proposition}[section]
\theoremstyle{definition}
\newtheorem{definition}[theorem]{Definition}

\theoremstyle{plain}
\newtheorem{maintheorem}{Theorem}

\newtheorem{remark}[theorem]{Remark}
\numberwithin{equation}{section}

\begin{document}

\title{Entropy stability for diffeomorphisms isotopic to Anosov on $\mathbb{T}^d$}

\author{Leonardo Parra}
\address{Instituto de Matem\'aticas, Pontificia Universidad Católica de Valpara\'iso, Chile.}
\email{leonardo.parra@pucv.cl}

\author{Sebasti\'an A. Ramirez}
\address{Facultad de Matem\'aticas, Pontificia Universidad Cat\'olica de Chile, Santiago, Chile}
\email{sramired@uc.cl}
\thanks{The second author was supported by Proyecto FONDECYT Postdoctorado 3240422, ANID, Chile.}

\author{Kendry J. Vivas}
\address{Departamento de Matem\'aticas, Universidad Católica del Norte, Antofagasta, Chile.}
\email{kendry.vivas01@ucn.cl}
\thanks{The third author was supported by  ANID Proyecto FONDECYT Inicicaci\'on 11250633, Chile.} 




\begin{abstract}
In this paper we study the behavior of topological entropy for partially hyperbolic diffeomorphisms on $\mathbb{T}^d$ isotopic to a linear Anosov automorphism with indecomposable weak-stable subspace. In particular, we prove that for a class of DA maps associated with such linear systems, whose central behavior is sufficiently dominated by the expansion rate of the linear model, the topological entropy coincides with that of the linear map. Furthermore, we obtain uniqueness of equilibrium states for a class of low-oscillation potentials and, as an application, the uniqueness of the measure of maximal entropy. We also provide sufficient conditions for which ergodic measures arise as the unique equilibrium state of some continuous potential. Finally, on $\mathbb{T}^3$ we construct a dynamically coherent diffeomorphism in the same isotopy class whose topological entropy is strictly larger than that of the linear part.
\end{abstract}

\maketitle

\section{Introduction and statements of the results}\label{statements}

Entropy is one of the most important invariants in dynamical systems. In simpler terms, this notion provides a quantitative measure of the complexity and unpredictability of orbits under iteration. Since the pioneering works of Adler, Konheim and McAndrew \cite{AKM}, Sinai \cite{S}, and Bowen \cite{B}, the role of entropy in hyperbolic dynamics has been well established, linking invariant measures, Lyapunov exponents, and the structural stability of uniformly hyperbolic systems. Among the most classical examples of the computation of topological entropy we found the Linear Anosov diffeomorphisms on tori, and their entropy is explicitly determined by the unstable eigenvalues of the underlying linear maps. They serve as natural models for exploring the persistence of dynamical invariants under perturbations.

An interesting topic about entropy theory is its preservation under deformations, {which provides} a way to find the ``simplest'' model {within} isotopy classes of diffeomorphisms on compact manifolds. In this research topic, the Nielsen-Thurston classification establishes that for any orientation preserving diffeomorphism on a compact surface is homotopic to either a periodic map ($g^p=id$ for some $p\in\mathbb{N}$), or a pseudo-Anosov, or leaves invariant some finite set of closed simple curves. In the first case, the entire homotopy class consists of systems with zero entropy. For pseudo-Anosov diffeomorphisms, it was shown in \cite{FS} that there exist diffeomorphisms $f$ isotopic to a pseudo-Anosov map $A$ whose topological entropy is greater than that of $A$. The third case was proven that can be reduced to the previous ones. Therefore, it is possible to identify {in a precise way} the simplest model in terms of entropy in any homotopy class of maps on compact surfaces. 

In higher dimensions, the problem of entropy preservation within isotopy classes becomes substantially more subtle. A fundamental source of examples beyond uniform hyperbolicity is given by the so-called Derived-from-Anosov (DA) diffeomorphisms on $\mathbb{T}^d$, first introduced by Mañé in his seminal work on the stability conjecture \cite{Ma}. These examples provided the first robustly transitive diffeomorphisms that are not Anosov, marking the emergence of partially hyperbolic dynamics as a central subject of study. Subsequent developments in the theory of stability and bifurcations, including, for instance,   contributions by Newhouse, Palis, and Takens in \cite{NP} and \cite{NPT}, further highlighted the richness of this setting.

In this context, {the authors in }\cite{CLPV} studied entropy preservation for DA diffeomorphisms on higher-dimensional tori and obtained sufficient conditions under which the topological entropy coincides with that of the linear Anosov representative. Besides, Theorem~B of that work illustrates the necessity of these hypotheses. More precisely, the authors construct a DA partially hyperbolic diffeomorphism $g$, isotopic to a linear Anosov map $A$, with an indecomposable two-dimensional central bundle and topological entropy strictly larger than that of $A$. Their construction was inspired by {a} model introduced by Bronzi and Tahzibi in \cite{BT}. However, the argument presented there contains certain gaps. This observation motivated us to revisit the problem and to establish the main results of the present paper.

In this article we show that the type of construction carried out in \cite{CLPV} may fail to increase entropy in certain situations.
More precisely, when the uniform unstable expansion is sufficiently large compared with the central expansion, the same construction as in \cite{CLPV} actually preserves the entropy of the linear part.
We also provide an alternative proof of their Theorem B, confirming that the statement of the theorem itself is correct. Our results provide an explicit criterion for entropy stability applicable to derived-from-Anosov diffeomorphisms with an indecomposable two-dimensional central bundle. In particular, Theorem~\ref{teo:MainTheorem} applies even to examples with positive topological entropy along some fibers of the semiconjugacy; see Proposition~\ref{example}.

To state our main results, we first describe the geometric setting of our perturbations. Let $A \colon \mathbb{T}^d \to \mathbb{T}^d$ be a linear Anosov automorphism admitting a dominated splitting of the form
\begin{displaymath}
T\mathbb{T}^d = E^{s}_A \oplus E^{ws}_A \oplus E^u_A,
\end{displaymath}
where $E^{s}_A \oplus E^{ws}_A$ corresponds to the stable bundle of $A$ and $E^u_A$ to its unstable bundle. We put $\lambda^{ws}_A=\|DA\mid_{E^{ws}_A}\|$ and $\lambda_A^u\!= m(DA_{E^{u}_A})$, where $m\left(L\right)$ denotes the co-norm of the linear map $L$. {Note that the strong stable bundle $E^{s}_A$ is allowed to be trivial.}

We consider partially hyperbolic systems obtained as $C^0$-small localized deformations of $A$ (see Definition \ref{def:WeakStableDef}). These deformations replace the uniformly contracting behavior along $E^{ws}_A$, turning it into a center bundle $E^c$ of dimension $d_c = \dim(E^{ws}_A)$, while preserving the dimensions of the strongly invariant directions.

Recall that a diffeomorphism $f \in \mathrm{Diff}^1(\mathbb{T}^d)$ is \emph{partially hyperbolic} if the tangent bundle admits a $Df$-invariant splitting $T\mathbb{T}^d  = E^s_f \oplus E^c_f \oplus E^u_f$ such that for every $x \in \mathbb{T}^d$,
\begin{equation*}
\|Df\!\mid_{E^s_f(x)}\| < m\left(Df\!\mid_{E^c_f(x)}\right) \quad \text{and}\quad \|Df\!\mid_{E^c_f(x)}\| < m\left(Df\!\mid_{E^u_f(x)}\right),
\end{equation*}
together with $\|Df|_{E^s_f}\| < 1$ and $\|Df^{-1}|_{E^u_f}\| < 1$.

In this context, we say that  a partially-hyperbolic diffeomorphism $f$ is \emph{dynamically coherent} if there exist $f$-invariant foliations $\mathcal{W}^{cs}$ and $\mathcal{W}^{cu}$ tangent to the distributions $E^s_f \oplus E^c_f$ and $E^c_f \oplus E^u_f$, respectively, whose intersection {defines} an invariant center foliation $\mathcal{W}^c$ tangent to $E^c_f$. In general, the center bundle $E^c_f$ need not be integrable when $d_c \geq 2$ (see, for instance, R. Hertz et al. \cite{RHRHU16}). Nevertheless, for diffeomorphisms isotopic to an Anosov automorphism $A$ on $\mathbb{T}^3$ and with a one-dimensional unstable direction, Potrie \cite{Po2015} showed that there is a unique invariant foliation $\mathcal{W}^{cs}$, having a global product structure property with the unstable foliation $\mathcal{W}^{u}$. 

Let $\mathrm{PH}_A(\mathbb{T}^d)$ denote the set of partially hyperbolic diffeomorphisms $f$ isotopic to $A$ such that $\dim(E^\sigma_f) = \dim(E^\sigma_A)$ for $\sigma \in \{s, u\}$. By a classical theorem of Franks \cite{Fr70}, every $f \in \mathrm{PH}_A(\mathbb{T}^d)$ admits a continuous surjective map $\pi_f \colon \mathbb{T}^d \to \mathbb{T}^d$ {such that}
\begin{equation*}
A \circ \pi_f = \pi_f \circ f.
\end{equation*}
Denote by $\mathrm{PH}^c_A(\mathbb{T}^d)$ the connected component of $\mathrm{PH}_A(\mathbb{T}^d)$ containing the linear model $A$. It is known that dynamical coherence is a $C^1$-open property \cite{FPS14} within this connected component.

{Next, we introduce the main definition of this paper:}

\begin{definition} \label{def:WeakStableDef}
Let $N\in\mathbb{N}$ and $r>0$. A $C^1$-diffeomorphism $g \colon \mathbb{T}^d \to \mathbb{T}^d$ is an $(N,r)$--\emph{localized deformation} of $A$ if $g \in \mathrm{PH}_A(\mathbb{T}^d)$ {and it differs from $A$ only inside the union of $N$ disjoint balls of radius $r$.} That is, there exist points $p_1, \dots, p_N \in \mathbb{T}^d$ such that:
\begin{enumerate}
    \item $B(p_i, r) \cap B(p_j, r) = \emptyset$ for $i \neq j$;
    \item $g(x) = A(x)$ for all $x \in \mathbb{T}^d \setminus \bigcup_{i=1}^N B(p_i, r)$.
\end{enumerate}
\end{definition}

 For $\varepsilon > 0$, we denote by $\mathcal{U}^0_\varepsilon(A)$ to a $C^0$-neighborhood of radius $\varepsilon$ around $A$. It is well known by \cite{W1} that if $\varepsilon$ is small enough, every map $g\in \mathcal{U}^0_\varepsilon(A)$ is semi-conjugate to $A$ by a map $\pi:\mathbb{T}^d\to\mathbb{T}^d$ which is $C^0$-close to the identity. 

Our first main result establishes the stability of the topological entropy of the linear model $A$ under {suitable} localized deformations. More precisely, we control the maximal expansion along the perturbed center direction inside the deformation region by a constant $\lambda_c>1$, while the contraction outside it is controlled by a bound $0<\lambda_s<1$, {in such a way that the topological entropy of the linear model is preserved.} The precise statement is as follows.
{
\begin{maintheorem} \label{teo:MainTheorem}
Let $A \colon \mathbb{T}^d \to \mathbb{T}^d$ be a linear Anosov diffeomorphism admitting a dominated splitting of the form
\begin{displaymath}
T\mathbb{T}^d = E^{s}_A \oplus E^{ws}_A \oplus E^u_A.
\end{displaymath} 
There exist a constant $C_1(A)>1$ with the following property: For every integer $N \geq 1$ and any choice of bounds $\lambda_c \in [1, C_1(A))$, there exist $r_1 > 0$, $\varepsilon > 0$ and a constant $C_2(A,\lambda_c)$ such that for any $(N,r_1)$-localized deformation $g_0\in \mathcal{U}^0_{\varepsilon}(A) \cap \mathrm{PH}^c_A(\mathbb{T}^d)$ satisfying
\begin{equation*} \label{ineq:Open0}
\sup_{x \in \mathcal{O}} \|Dg_0|_{E^c_{g_0}(x)}\| < \lambda_c \quad \text{and} \quad \sup_{x \in \mathbb{T}^d \setminus \mathcal{O}} \|Dg_0|_{E^c_{g_0}(x)}\| < \lambda_s,
\end{equation*}
where $\mathcal{O} = \bigcup_{i=1}^N B(p_i, r_1)$ and $\lambda_s \in [\lambda_A^{ws}, C_2(A,\lambda_c))$, there exists a $C^1$-neighborhood $\mathcal{U} \subset \mathcal{U}^0_{\varepsilon}(A)$ of $g_0$ such that 
\begin{equation*}
h_{\mathrm{top}}(g) = h_{\mathrm{top}}(A),\quad\forall g\in\mathcal{U}. 
\end{equation*}
\end{maintheorem}}

\begin{remark}
Theorem~\ref{teo:MainTheorem} applies in particular to the derived-from-Anosov maps considered by Buzzi et al. in \cite{BFSV2012}.
\end{remark}

Now, recall that a diffeomorphism $f$ is \textit{$C^1$-robustly transitive} if there is a $C^1$-neighborhood $\mathcal{V}$ of $f$ such that any $g\in\mathcal{V}$ is transitive. The second main result of this article shows that, when the central bundle is indecomposable, it is possible to construct a robustly transitive diffeomorphism in the isotopy class of a linear Anosov diffeomorphism $A$ whose topological entropy is strictly larger than that of $A$. 
 \begin{maintheorem}\label{theorem B} There exists a robustly transitive partially hyperbolic DA diffeomorphism $f$ isotopic to a linear Anosov automorphism $A:\mathbb{T}^3\to \mathbb{T}^3$ with indecomposable two-dimensional stable bundle  such that $h_{\text{top}}(f) > h_{\text{top}}(A)$.
 \end{maintheorem}


Once the topological complexity of these perturbations is understood according to Theorem \ref{teo:MainTheorem}, a natural question that arises concerns the statistical properties of the system. For this, we briefly recall that for a $C^1$-diffeomorphism $g:\mathbb{T}^d\to \mathbb{T}^d$ and a continuous potential $\phi:\mathbb{T}^d\to \mathbb{R}$, the topological pressure of the pair $(g,\phi)$ is defined by the variational principle \cite{W75}:
$$
P_{\text{top}}(g, \phi) = \sup_{\mu \in \mathrm{Prob}(g)} \left\{ h_\mu(g) + \int \phi \, d\mu \right\}.
$$
A measure attaining this supremum is called an \emph{equilibrium state} for $(g, \phi)$.

A central problem is to find conditions ensuring the existence and uniqueness of equilibrium states for continuous potentials. For uniformly hyperbolic systems, the thermodynamic formalism was developed by Sinai~\cite{Sin72}, and later by Ruelle and Bowen ~\cite{B74,R}. However, extending these results to partially hyperbolic diffeomorphisms presents substantial difficulties due to the non-uniform behavior along the central direction.

Important progress in the partially-hyperbolic setting was achieved by Climenhaga and Thompson~\cite{CT16}, who developed a general framework to prove intrinsic ergodicity by decomposing the dynamics into components with specification properties and components with low entropy. Buzzi and Fisher~\cite{BF} obtained a related entropy-conjugacy result for $C^1$-large, $C^0$-small deformations of Anosov diffeomorphisms admitting only a dominated splitting, establishing a partial conjugacy that identifies invariant measures of near-maximal entropy. In the specific case of Derived-from-Anosov systems on $\mathbb{T}^3$, Buzzi et al. in \cite{BFSV2012} established the uniqueness of the measure of maximal entropy by exploiting the one-dimensional structure of the center bundle.

A complementary structural result concerns the existence of equilibrium states in partially hyperbolic dynamics. In the presence of a dominated splitting of the central bundle into one-dimensional subbundles, Díaz et al. in \cite{DFPV} proved that such systems are entropy-expansive. As a consequence, the metric entropy is upper semicontinuous with respect to the weak$^\ast$ topology on the set of invariant measures, which in turn guarantees the existence of measures of maximal entropy and, more generally, equilibrium states for every continuous potential via classical results of Misiurewicz~\cite{Mi}. This structural mechanism provides a general theory in the one-dimensional center setting and underlies the applicability of the results in~\cite{CT}.

In contrast, when the center bundle is two-dimensional and indecomposable, entropy-expansiveness is not known to hold in general, and the above existence mechanism breaks down. Consequently, the existence and uniqueness of equilibrium states in this setting require new arguments beyond the standard entropy-expansive framework.

For non-zero potentials, uniqueness of equilibrium states often requires additional assumptions, such as a low oscillation condition ensuring that the pressure is dominated by the hyperbolic behavior rather than by entropy generated along the fibers of the semiconjugacy. In this context, Crisostomo and Tahzibi in \cite{CT} proved that for any diffeomorphism isotopic to a linear Anosov map with a one-dimensional center bundle, a unique equilibrium state exists whenever the variation of the potential is sufficiently small compared to the topological entropy. Their result applies to the entire isotopy class, but relies essentially on the one-dimensional nature of the center.

In this paper, we treat with a two-dimensional indecomposable central bundle. The absence of a finer invariant splitting may allow for richer internal dynamics, potentially supporting multiple equilibrium states. By considering lifted potentials of the form $\varphi=\psi\circ\pi$ through the semiconjugacy $\pi$, we show that a suitable pressure gap between the linear model and the entropy produced along the fibers is sufficient to guarantee uniqueness in this scenario.

In the setting of Theorem~\ref{teo:MainTheorem}, the semiconjugacy $\pi:\mathbb{T}^d\to\mathbb{T}^d$ between $g$ and $A$ provides a natural mechanism to transfer potentials from the linear model to the perturbation. To this end, let $\psi:\mathbb{T}^d\to\mathbb{R}$ be a continuous potential such that $(A,\psi)$ admits a unique equilibrium state $\mu$, satisfying the following low oscillation condition:
\begin{equation}\label{eq:oscillation-condition}
   \sup_{\mathbb{T}^d} \psi - \inf_{\mathbb{T}^d} \psi < h_{\mathrm{top}}(A) - h_1,
\end{equation}
where $h_1<h_{\mathrm{top}}(A)$ is a constant to be specified in Proposition~\ref{pro:Main2}. The next result of this paper is the following:

\begin{maintheorem}\label{coroequilibrium}
The potential $\varphi := \psi \circ \pi$, where $\psi$ satisfies \eqref{eq:oscillation-condition}, admits a unique equilibrium state $\nu$, which is hyperbolic, that is, it has non-zero Lyapunov exponents. In particular, $g$ admits a unique measure of maximal entropy.
\end{maintheorem}


The final contribution of this article addresses the \textit{inverse problem of thermodynamic formalism} in the spirit of the classification of invariant measures. We aim to determine which measures $\nu$ can be realized as unique equilibrium states for some continuous potential. A closely related question for which measures $\nu$ \emph{some} continuous potential exists for which $\nu$ is an equilibrium state, without uniqueness has recently been fully characterized by Hedges~\cite{Hed26} via upper semicontinuity of the entropy map at $\nu$. Theorem~\ref{theorem D} below instead addresses the stronger requirement of uniqueness.

Let $h_0<h_1$ be the constants given by Proposition~\ref{pro:Main2}, which satisfy $h_1=h_0+d_c\log(\lambda_c)$. Denote by $\mathrm{Prob}^{h_0}_{\mathrm{erg}}(A)$ the set of ergodic measures $\mu$ for $A$ satisfying $h_{\mu}(A)>h_0$. For a given $\mu \in \mathrm{Prob}^{h_0}_{\mathrm{erg}}(A)$, let $\psi_\mu$ be the potential provided by Phelps~\cite{P} for which $\mu$ is the unique equilibrium state. To ensure that this state is not ``submerged" by the fiber entropy, we quantify its isolation via the \textit{pressure gap}, defined by
$$
\Delta_\mu := P_{\mathrm{top}}(A, \psi_\mu) - \sup \{ P_{\zeta}(A, \psi_\mu) : h_{\zeta}(A) \le h_0 \}.
$$

Note that $\Delta_\mu>0$ automatically: the set $\{\zeta\in\mathrm{Prob}(A): h_\zeta(A)\le h_0\}$ is compact, by upper semicontinuity of the entropy map for the expansive system $A$, and does not contain $\mu$; since $\mu$ is the unique equilibrium state of $(A,\psi_\mu)$, the supremum defining $P_\zeta(A,\psi_\mu)$ over that set is strictly smaller than $P_{\mathrm{top}}(A,\psi_\mu)$.

\begin{maintheorem}\label{theorem D}
Let $g:\mathbb T^d\to\mathbb T^d$ be a diffeomorphism satisfying the hypotheses of Theorem~\ref{teo:MainTheorem}, with central expansion $\lambda_c$, and set $d_c:=\dim(E^{c}_g)$. Let $\nu \in \mathrm{Prob}^{h_1}_{\mathrm{erg}}(g)$ and set $\mu = \pi_* \nu$. If
\begin{equation}\label{eq:fiber-condition-refined}
d_c\log(\lambda_c) < \Delta_\mu,
\end{equation}
then there exists a continuous potential $\varphi: M \to \mathbb{R}$ such that $\nu$ is the \emph{unique} equilibrium state of $(g, \varphi)$.
\end{maintheorem}

\begin{remark}
Condition~\eqref{eq:fiber-condition-refined} controls the fiber complexity $d_c\log(\lambda_c)$ introduced by the perturbation. Unlike the uniqueness result of Buzzi et al. \cite{BFSV2012}, which requires vanishing fiber entropy, Theorem~\ref{theorem D} allows positive fiber entropy provided it is dominated by the pressure gap $\Delta_\mu$ as occurs, for instance, in the example of Proposition~\ref{example}. Since $\Delta_\mu>0$ for every  $\mu\in\mathrm{Prob}^{h_0}_{\mathrm{erg}}(A)$ and $\lambda_c$ can be taken arbitrarily close to 1 by shrinking the $C^1$-neighborhood of $A$ in Theorem~\ref{teo:MainTheorem}, condition~\eqref{eq:fiber-condition-refined} holds for every such $\mu$. 
\end{remark}

\section{Proof of Theorem \ref{teo:MainTheorem}}

In this section, {we present the proof of Theorem \ref{teo:MainTheorem}. For this, we first establish some technical results from which Theorem~\ref{teo:MainTheorem} follows.}

{The first subsection is devoted to} establish a rigidity principle: invariant measures with entropy close to $h_{\mathrm{top}}(A)$ cannot concentrate a definite proportion of their mass in arbitrarily small regions. More precisely, we show that, within a sufficiently small $C^0$-neighborhood of a linear Anosov diffeomorphism $A$, invariant ergodic measures with high entropy (close to $h_{\mathrm{top}}(A)$) cannot concentrate their mass inside balls of small radius. This property is inherited from $A$ via topological stability, and will be used to prevent the appearance of invariant measures with entropy larger than $h_{\mathrm{top}}(A)$ under the partially hyperbolic perturbations considered in this work.

Meanwhile, in the second subsection, we establish the precise conditions for which the rigidity of the topological entropy is obtained. Concretely, for the class of $(N,r)$--\emph{localized deformations} of $A$, {we will see that} this fact follows from controlling the entropy that the central direction can support and the uniform growth rate of the unstable direction. {This help us, in particular, to get a  control of the}  topological entropy that the fibers can carry.

\subsection{Robust non--concentration.}\label{subsec:MeasureSmall} 

{In this subsection we prove that, in a small $C^0$–neighborhood of an Anosov diffeomorphism, the non-concentration property for invariant measures with entropy close to the maximum persists uniformly. Both, the statement and its proof, are formulated in the general setting of Anosov diffeomorphisms on compact manifolds.} 

The main result of this subsection is the following:

\begin{propo}\label{prop:Main1}
Let $A: M\to M$ be an Anosov diffeomorphism  on a compact Riemannian manifold $M$. Let $\eta\in (0,1)$, and let 
$$
h_0\in \big(h_{\mathrm{top}}(A)(1-\eta),\, h_{\mathrm{top}}(A)\big).
$$
Then, there exist constants $r_1>0$ and $\varepsilon>0$ such that for every homeomorphism $g \in \mathcal{U}^0_{\varepsilon}(A)$ and every $\nu \in \operatorname{Prob}_{\mathrm{erg}}(g)$ satisfying $\pi_* \nu \in \operatorname{Prob}_{\mathrm{erg}}^{h_0}(A)$, one has
$$
\nu\big(B(x,r)\big) < \eta, \quad \forall x \in M, \;\forall r\in (0,r_1].
$$
\end{propo}

The proof is obtained by combining the non-concentration property for Anosov diffeomorphisms due to Buzzi--Fisher \cite{BF} with the semiconjugacy provided by topological stability proven by Walters in \cite{W1}.

In order to prove the above proposition, we first recall the result of Buzzi--Fisher.

\begin{lemma}[Lemma 5.1 in \cite{BF}]\label{le:BuzziFisher} 
Let $A: M\to M$ be an Anosov diffeomorphism. For every $\eta > 0$ and every $h_0 \in \big(h_{\mathrm{top}}(A)(1-\eta), h_{\mathrm{top}}(A)\big)$, there exists $r_0 > 0$ such that for every $\mu \in \mathrm{Prob}_{\mathrm{erg}}^{h_0}(A)$, one has $\mu\big(B(x, r_0)\big) < \eta$ for all $x \in M$.
\end{lemma}

We also recall that Anosov diffeomorphisms are topologically stable. 

\begin{lemma}\label{le:TopEstab} For every Anosov diffeomorphism $A : M \to M$ there exists $\delta_0 > 0$ such that for every $0 < \delta < \delta_0$, there exists $\varepsilon > 0$ with the following property: if $g : M \to M$ satisfies $\text{dist}_{C^0}(g, A) < {\varepsilon}$, then there exists a unique continuous and surjective map $\pi : M \to M$ such that 
$$
\pi \circ g = A \circ \pi\quad \text{and} \quad \text{dist}_{C^0}(\pi, \mathrm{Id}) < \delta.
$$
\end{lemma}

Next, we prove the main result of this subsection.

\begin{proof}[Proof of Proposition \ref{prop:Main1}] 
Fix $\eta>0$ and $h_0\in \big(h_{\mathrm{top}}(A)(1-\eta),\, h_{\mathrm{top}}(A)\big)$, and let $r_0>0$ be given by Lemma \ref{le:BuzziFisher}.

Let {$0<\delta< \min\lbrace \delta_0, r_0/4\rbrace$}. By Lemma \ref{le:TopEstab} there exists $\varepsilon > 0$ such that every homeomorphism $g \in \mathcal{U}_{\varepsilon}^0(A)$ is semi-conjugate to $A$ via a continuous surjective map $\pi : M \to M$ satisfying $\operatorname{dist}_{C^0}(\pi, \mathrm{Id}) < \delta$. In particular, for every invariant probability measure $\nu \in \mathrm{Prob}_{\mathrm{erg}}(g)$ with $\mu := \pi_*\nu \in \mathrm{Prob}_{\mathrm{erg}}^{h_0}(A)$, one has by Lemma~\ref{le:BuzziFisher} that $\mu\big(B(\pi(x), r_0)\big) < \eta$ for all $x \in M$.

Define $r_1 := r_0 - 2\delta$. Since $\delta < r_0/4$, we have $r_1 > r_0/2 > 0$. We claim that for every $x\in M$ and $r\in (0,r_1]$, the inclusion $B(x, r) \subset \pi^{-1}\big(B(\pi(x), r_0)\big)$ holds. Indeed, for any $y \in B(x, r)$, the triangle inequality yields
\begin{align*}
\operatorname{dist}\big(\pi(y), \pi(x)\big) &\leq \operatorname{dist}\big(\pi(y), y\big) + \operatorname{dist}(y, x) + \operatorname{dist}\big(x, \pi(x)\big) \\
&< \delta + r + \delta \leq r_1 + 2\delta = r_0.
\end{align*}
Thus, $\pi(y) \in B(\pi(x), r_0)$. Therefore,
$$
\nu\big(B(x, r)\big) \leq \nu\Big(\pi^{-1}\big(B(\pi(x), r_0)\big)\Big) = \mu\big(B(\pi(x), r_0)\big) < \eta,
$$
which concludes the proof.
\end{proof}

\subsection{Negative central Lyapunov exponents}
Recall that for any point $x\in \mathbb{T}^d$ we define the maximal asymptotic expansion along the center bundle as
$$
\lambda^{c}(g,x) = \limsup_{n \to \infty} \frac{1}{n}\log\left(\left\|Dg^n(x)\mid_{E^{c}_g(x)}\right\|\right).
$$
By Oseledets' Multiplicative Ergodic Theorem (see \cite[Theorem S.2.9]{HK95}), the Lyapunov exponents are well-defined for $\mu$-almost every point. Furthermore, since $\mu$ is ergodic, these exponents are constant almost everywhere.

In this subsection, we prove that the Lyapunov exponents along the central direction are negative for ergodic measures with large entropy for diffeomorphisms in a $C^1$-neighborhood of a $(N,r)$-localized weak deformation $g_0$ that is $C^0$-close enough to $A$, even though the $C^1$-distance to $A$ may be large. As we will see in Section \ref{constructionCarrasco}, this separation in the $C^1$-topology allows the construction of partially hyperbolic diffeomorphisms whose center-stable leaves carry positive topological entropy, while the total topological entropy of the system remains equal to $h_{\mathrm{top}}(A)$.

Recall that $A \colon \mathbb{T}^d \to \mathbb{T}^d$ denotes a linear Anosov automorphism admitting a dominated splitting of the form
\begin{equation*}
T\mathbb{T}^d = E^{s}_A \oplus E^{ws}_A \oplus E^u_A,
\end{equation*}
where $E^{s}_A \oplus E^{ws}_A$ is the stable bundle of $A$ with $\lambda_A^{ws}=\|DA\mid_{E^{ws}_A}\|$, and $E^u_A$ is the unstable bundle with $\lambda_A^u\!:= m(DA_{E^{u}_A})$. The main result of this section is the following proposition:

\begin{propo}\label{pro:Main2}
{Let $A\colon \mathbb{T}^d \to \mathbb{T}^d$ be a linear Anosov automorphism with splitting $T\mathbb{T}^d = E^{s}_A \oplus E^{ws}_A \oplus E^u_A$. Let consider the constant
\begin{displaymath}
    C_1(A)=\exp\left(\frac{\log(\lambda_A^{ws}) + \sqrt{\log(\lambda_A^{ws}) \left[ \log(\lambda_A^{ws}) - \frac{4h_{\mathrm{top}}(A)}{d} \right]}}{2}\right).
\end{displaymath}
Fix $\lambda_c\in \left[1,\, \min\left\{C_1(A),\, \lambda_A^u\right\}\right)$, and let 
\begin{displaymath}
    C_2(A,\lambda_c)= \exp\left(\frac{d_c\log^2(\lambda_c)}{d_c\log(\lambda_c)-h_{\mathrm{top}}(A)}\right),
\end{displaymath}
where $d_c=\dim(E^{ws}_A)$. Take $\lambda_s\in \left(\lambda_A^{ws},\,C_2(A,\lambda_c)\right)$. Then, for each $N\geq 1$, there exist constants $r_1>0$ and $\varepsilon>0$ such that for every $0<r\leq r_1$ and every $(N,r)$-localized deformation $g_0$ of $A$ in $\mathcal{U}_{\varepsilon}^0(A)\cap \mathrm{PH}_A^c(\mathbb{T}^d)$ satisfying
\begin{equation}\label{ineq:Open}
\sup_{x\in \mathcal{O}}\left\|Dg_0(x)\mid_{E^c_{g_0}(x)}\right\|< \lambda_c, \qquad  \sup_{x\in \mathbb{T}^d\setminus \mathcal{O}} \left\|Dg_0(x)\mid_{E^c_{g_0}(x)}\right\|< \lambda_{s},
\end{equation}
there exist $h_1>0$ and a $C^1$--neighborhood
$\mathcal U'$ of $g_0$
such that for every $g\in\mathcal U'$
and every
$\nu\in \mathrm{Prob}^{h_1}_{\mathrm{erg}}(g)$, the $\nu$-center Lyapunov exponents are negative.}
\end{propo}

{To prove this proposition, we first establish a control over invariant measures with large entropy. More precisely, we have the following lemma.}

\begin{lemma}\label{le:keylemma}
Let $0<\delta<\delta_0$ and $\varepsilon>0$ given by Lemma \ref{le:TopEstab}. Assume that $g_0\in \mathcal{U}^0_{\varepsilon}(A)\cap \mathrm{PH}^c_A(\mathbb{T}^d)$. Then, for every $h_0>0$, there exist $h_1>h_0$ such that
$$
\pi_*\nu \in \mathrm{Prob}_{erg}^{{h}_0}(A)
$$
for every $\nu \in \mathrm{Prob}^{{h}_1}_{erg}(g_0)$. 
\end{lemma}

\begin{proof} Define the function $\lambda_c(\cdot):\mathbb{T}^d\to (0,+\infty)$ as
\begin{equation}\label{eq:SpecGap}
\lambda_c(x):=\left\|Dg_0(x)|_{E^{c}_{g_0}(x)}\right\|.   
\end{equation}
It is clear that $\lambda_c$ is continuous from partial-hyperbolicty of $g_0$.  
Fix $h_0>0$ and choose $h_{1} \geq h_0 + \dim(E^{c}_g)\log\left(\max_{x\in \mathbb{T}^d}\lambda_c(x)\right)$. 

Since $g_0\in\mathcal{U}^0_{\varepsilon}(A)$, we have
by Lemma \ref{le:TopEstab} that the maps $g$ and $A$ are semi-conjugate by a continuous surjective map $\pi$ which is $\delta$-close to the identity. On the other hand, since $g_0\in\mathrm{PH}^c_A(\mathbb{T}^d)$, it is dynamically coherent, i.e., it admits an invariant foliation of $\mathbb{T}^d$ by central leaves. 

Notice that for every $x\in \mathbb{T}^d$, the fiber $\pi^{-1}(\pi(x))$ is contained in a center-stable leaf of $g$. Indeed, since $E^u_g$ is uniformly expanding, $\pi^{-1}(\pi(x))$ has no component along $E^u_g$. Otherwise, if $y\in\pi^{-1}(\pi(x))$ has component in $E_g^u$, then  $d(g^n(x),g^n(y))\to\infty$, contradicting the uniform bound $\mathrm{diam}(\pi^{-1}(z))\le2\delta$ given by Lemma~\ref{le:TopEstab}. Likewise, $\pi^{-1}(\pi(x))$ has no component along $E^s_g$. Otherwise, if $y\in\pi^{-1}(\pi(x))$ had a nonzero component along $E^s_g$, then $d(g^{-n}(x),g^{-n}(y))\to\infty$ because $E^s_g$ is uniformly expanding under $g^{-1}$. This contradicts the same uniform bound of Lemma~\ref{le:TopEstab}, now for $n\le0$. Hence, $$\pi^{-1}(\pi(x))\subset\mathcal W^c_g(x),\quad x\in\mathbb T^d.$$ 

In this way, by applying the standard upper bound for the topological entropy of a smooth map on a submanifold of dimension $\dim(E^{c}_{g_0})$, we obtain
$$
{h}_{\mathrm{top}}\big(g_0,\pi^{-1}(\pi(x))\big) \leq \dim(E^{c}_{g_0}) \log\left(\sup_{x\in \mathbb{T}^d}\|Dg_0|_{E^{c}_{g_0}(x)}\|\right).
$$

By the Ledrappier-Walters formula for semiconjugacies \cite{LW}, we have for any $\nu\in \mathrm{Prob}_{\mathrm{erg}}^{h_1}(g)$ that
$$
h_1 < {h}_{\nu}(g_0) \leq {h}_{\pi_*\nu}(A) + \int_{\mathbb{T}^d} h_{\mathrm{top}}\big(g_0,\pi^{-1}(\pi(x))\big) \, d\nu(x).
$$
By taking the supremum over all fibers, we get
$$
h_1 < {h}_{\pi_*\nu}(A) + \sup_{x\in \mathbb{T}^d} h_{\mathrm{top}}\big(g,\pi^{-1}(\pi(x))\big) \leq {h}_{\pi_*\nu}(A) + \dim(E^{c}_g)\log\left(\max_{x\in \mathbb{T}^d}\lambda_c(x)\right).
$$
By our choice of $h_1$, this directly implies ${h}_{\pi_*\nu}(A) > h_0$. Thus, $\pi_*\nu \in \mathrm{Prob}_{\mathrm{erg}}^{{h}_0}(A)$, concluding the proof. 
\end{proof}

\begin{remark}\label{robustcoherence}
It should be noticed that, in the above lemma, the dynamical coherence of $g_0\in \mathcal{U}_{\varepsilon}^0(A)\cap \mathrm{PH}_A^c(\mathbb{T}^d)$ is, actually, the crucial fact for the proof of the result. In this way, since partial hyperbolicity and dynamic coherence are $C^1$-robust (see \cite[Theorem A]{FPS14}), there exists a $C^1$-neighborhood $\mathcal{U}'\subset \mathcal{U}^0_{\varepsilon}(A)$ of $g_0$ such that Lemma \ref{le:keylemma} holds for any $g\in\mathcal{U}'$.   
\end{remark}

The following technical lemma establishes the parameter hierarchy required to ensure a negative central Lyapunov exponent while preserving the topological entropy of the linear model $A$. These constants must be chosen sequentially. First, the linear Anosov automorphism $A$ gives a constrain in the maximum central expansion $\lambda_c$ inside the perturbation. This $\lambda_c$, in turn, dictates the compensating contraction $\lambda_s$ required outside the support. 

The lemma establishes that for any such pair $(\lambda_c, \lambda_s)$, there exists a valid open interval $I$ such that by choosing the mass-concentration bound $\eta \in I$ simultaneously resolves the geometric and dynamical constraints: it guarantees a strictly negative average central expansion ($\eta\log\lambda_c + (1-\eta)\log\lambda_s < 0$), while providing the exact entropy thresholds $h_0 < h_1 < h_{\mathrm{top}}(A)$. This strict upper bound is dynamically necessary; if we allowed $h_1 \geq h_{\mathrm{top}}(A)$, the set of measures $\mathrm{Prob}_{\mathrm{erg}}^{h_1}(g)$ could be empty, rendering the subsequent arguments vacuous. On the other hand, the lower bound of $I$ explicitly prevents this topological overflow, showing why $\eta$ cannot be chosen arbitrarily close to zero.

\begin{lemma}\label{le:Aux}
Let $ \tau > 0$, $d\geq 1$, $s\in (0,1)$, and 
$$
b\in \left[1,\, \exp\left(\frac{\log(s) + \sqrt{\log(s) \left[ \log(s) - \frac{4\tau}{d} \right]}}{2}\right)\right).
$$
Then, the following holds:
\begin{enumerate}
\item The interval $I_1=\left(\frac{d\log(b)}{\tau},\,1\right)$ is well defined. Moreover, for every $\eta\in I_1$ there is
$\tilde{\varepsilon}>0$ sufficiently small such that the numbers
$
h_0 := \tau(1 - \eta)+\tilde{\varepsilon}$ and $h_1 := h_0+ d\log(b)
$ satisfy
\begin{equation}\label{eq:EntroAlta}
\tau\left(1-\eta\right)<h_0<h_1<\tau.
\end{equation}
\item For every $a \in \left[s,\, \exp\left({\frac{d\log^2(b)}{d\log(b)-\tau}}\right)\right)$, the interval $I_2:= \left(\frac{d\log(b)}{\tau},\, \frac{\log(a)}{\log(a b^{-1})}\right)$ is nonempty and it is contained in $I_1$. Moreover, for every $\eta\in I_2$:
\begin{enumerate}
\item[(i)] $b^{\eta} a^{1-\eta}<1$;
\item[(ii)] there exist $\tilde{\varepsilon}>0$ such that $h_0$ and $h_1$, as in item (1), satisfy the relation \eqref{eq:EntroAlta}.
\end{enumerate}
\end{enumerate}
\end{lemma}

\begin{remark}
The domain for the parameter $b$ is well defined. Indeed, it is easy to check that 
$$
\log(s)+\sqrt{\log(s) \left[ \log(s) - \frac{4\tau}{d} \right]}>0, \quad \forall s\in (0,1).
$$
\end{remark}

\begin{proof} 
Notice that  
\begin{displaymath}
   \log(s)\left(\log(s)-\frac{4\tau}{d}\right)<\left(\frac{2\tau}{d}\right)^2+\log(s)\left(\log(s)-\frac{4\tau}{d}\right)=\left(\frac{2\tau}{d}-\log(s)\right)^2, 
\end{displaymath}
so that 
\begin{displaymath}
\frac{d\log(b)}{\tau}<\frac{d}{\tau}\cdot\left(\frac{\log(s)+\sqrt{\log(s) \left[ \log(s) - \frac{4\tau}{d} \right]}}{2}\right)<1.    
\end{displaymath}
This shows that $I_1$ is well defined. 

Fix $\eta\in \left(\frac{d\log(b)}{\tau},\,1\right)$ and choose $\tilde{\varepsilon}>0$ so that
$$
\frac{\tilde{\varepsilon}+d\log(b)}{\tau}<\eta<1.
$$
By the definitions of $h_0$ and $h_1$, the first two inequalities in  \eqref{eq:EntroAlta} are immediate. For the last inequality, the choice of $\tilde{\varepsilon}$ implies that
$$
\tilde{\varepsilon}+d\log(b)<\eta \tau \ \Leftrightarrow \ h_1=\tau(1 - \eta) +\tilde{\varepsilon}+ d \log(b)<\tau,
$$
which  proves item \emph{(1)}.

Now, since $\frac{d\log(b)}{\tau}<1$ we obtain that $c:=e^{\frac{d\log^2(b)}{d\log(b)-\tau}}<1$. In this way, for $a\in [s,\, c)$ we have
$$
\frac{d\log(b)}{\tau}<\frac{\log(s)}{\log(s b^{-1})}\leq \frac{\log(a)}{\log(a b^{-1})}.
$$
Hence, the interval $I_2$ is nonempty. Moreover, for $a\in[s,c)$ and $\eta\in I_2$, one has 
$$
\eta<\frac{\log(a)}{\log(a b^{-1})} \quad  \Leftrightarrow \quad  (1-\eta)\log(a)+\eta\log(b)<0,
$$
which implies \emph{(i)}. Since $\eta\in I_2\subset I_1$, we directly deduce \emph{(ii)}.

Thus, we conclude the proof.
\end{proof}

As a consequence of Lemma \ref{le:keylemma} and Lemma \ref{le:Aux}, we have

\begin{propo}\label{cor:Main1}
Let  $C_1(A)>1$ is as in Proposition \ref{pro:Main2}. Let 
$\lambda_c\in \left[1,\,C_1(A)\right)$ and $\eta \in \big(\frac{d_c\log(\lambda_c)}{h_{\mathrm{top}}(A)},1\big)$. There exist constants 
$
h_{\mathrm{top}}(A)(1-\eta)<h_0<h_1<h_{\mathrm{top}}(A),
$
and positive numbers ${r}_1,{\varepsilon}$ such that if $g\in \mathcal{U}^0_{{\varepsilon}}(A)\cap \mathrm{PH}^c_A(\mathbb{T}^d)$ verifies 
$$
\sup_{x\in \mathbb{T}^d}\lambda_c(x)\leq \lambda_c,
$$
then for every
$\nu\in \mathrm{Prob}^{h_1}_{\mathrm{erg}}(g)$ and every $r\in (0,r_1]$, one has
$$ \pi_*\nu \in \mathrm{Prob}^{h_0}_{\mathrm{erg}}(A) \quad{and} \quad \nu(B(x,{r}))<{\eta}, \quad  \forall \ x\in\mathbb{T}^d.
$$
\end{propo}

\begin{proof} Apply Lemma \ref{le:Aux} with $\tau=h_{\mathrm{top}}(A)$, $d=d_c$, $b=\lambda_c$, and $\eta \in \left(\frac{d_c\log(\lambda_c)}{h_{\mathrm{top}}(A)},1\right)$. Thus, by item \emph{(1)} of that lemma, there is $\tilde{\varepsilon}>0$ such that the constants
$h_0 := h_{\mathrm{top}}(A)(1-\eta) + \tilde{\varepsilon}$ and $h_1 := h_0 + d_c\log(\lambda_c)$ satisfy $$h_{\mathrm{top}}(A)(1-\eta) <h_0 <h_1 < h_{\mathrm{top}}(A).$$
For the parameters $\eta$ and $h_0$ above, Proposition \ref{prop:Main1} provides $r_1>0$ and $\varepsilon>0$.

Now, consider a diffeomorphism $g \in\mathcal{U}^0_{\varepsilon}(A)\cap \mathrm{PH}^c_A(\mathbb{T}^d)$ with $\sup_{x\in \mathbb{T}^d}\lambda_c(x)\leq \lambda_c
$, and let $\nu \in \mathrm{Prob}^{h_1}_{\mathrm{erg}}(g)$. Then, $\pi_*\nu \in \mathrm{Prob}^{\hat{h}_0}_{\mathrm{erg}}(A)$ as a consequence of Lemma \ref{le:keylemma}. Hence, by Proposition \ref{prop:Main1} we have for every $r \in (0, r_1]$ that
$$
\nu\big(B(x,r)\big) < \eta, \quad \text{for all } x\in\mathbb{T}^d.
$$
This completes the proof.
\end{proof}

Next, we are ready to prove Proposition \ref{pro:Main2}. 

\begin{proof}[proof of Proposition \ref{pro:Main2}]
    Let $N\geq 1$. From the choice of $\lambda_c$, take $\eta \in \big(\frac{d_c\log(\lambda_c)}{h_{\mathrm{top}}(A)},1\big)$. Notice that this choice of $\eta$ implies 
    \begin{equation}\label{inq:ExpNeg}
\eta\log(\lambda_c) + (1-\eta)\log(\lambda_s) < 0.
\end{equation}
Let consider the constants $h_{\mathrm{top}}(A)(1-\eta)<h_0<h_1<h_{\mathrm{top}}(A)$, and the positive numbers $r_1$ and $\varepsilon$ given by Proposition \ref{cor:Main1}.

Let $0<r<r_1$ and let $g_0$ be a $(N,r)$-localized deformation of $A$ in $\mathcal{U}_{\varepsilon}^0(A)\cap \mathrm{PH}_A^c(\mathbb{T}^d)$ satisfying
\begin{displaymath}
\sup_{x\in \mathcal{O}}\left\|Dg_0(x)\mid_{E^c_{g_0}(x)}\right\|< \lambda_c, \qquad  \sup_{x\in \mathbb{T}^d\setminus \mathcal{O}} \left\|Dg_0(x)\mid_{E^c_{g_0}(x)}\right\|< \lambda_{s}. 
\end{displaymath}
By Proposition \ref{cor:Main1} and Remark \ref{robustcoherence} there exists a $C^1$-neighborhood $\mathcal{U}'\subset \mathcal{U}^0_{\varepsilon}(A)$ of $g_0$ such that 
 the inequalities \eqref{ineq:Open} persist on a continuation set $\mathcal{O}_g$ for $g\in\mathcal{U}'$ still covered by the $N$ balls $B(p_i, r)$, $i=1,\ldots,N$, so that
  $$
\nu(\mathcal{O}_g) \leq \sum_{i=1}^N \nu(B(p_i, r)) < \eta.
$$
for every $\nu\in \mathrm{Prob}_{\mathrm{erg}}^{h_1}(g)$. 

Next, we evaluate the upper center-stable Lyapunov exponent $\lambda^{cs}(g,x)$ for any $\nu\in \mathrm{Prob}_{\mathrm{erg}}^{h_1}(g)$. For every $x \in \mathbb{T}^d$ and $n \geq 1$, let $k_n(x)$ be the number of iterates $0 \le j < n$ such that $g^j(x) \in \mathcal{O}_g$. By Birkhoff's ergodic theorem, for $\nu$-almost every $x$, the asymptotic frequency of visits to $\mathcal{O}_g$ satisfies
$$
\lim_{n\to\infty} \frac{k_n(x)}{n} = \nu(\mathcal{O}_g) < \eta.
$$
By using the submultiplicativity of the derivative and the robust upper bounds on $\mathcal{O}_g$ and its complement, we can estimate the expansion along the center-stable bundle as
$$
\big\|Dg^n|_{E^{c}_g(x)}\big\| \leq \lambda_c^{k_n(x)} \lambda_s^{n-k_n(x)}.
$$
So, by taking logarithms, dividing by $n$, and passing to the limit, we obtain for $\nu$-almost every $x\in\mathbb{T}^d$ that 
\begin{align*}
\lambda^{cs}(g,x) &= \limsup_{n\to \infty} \frac{1}{n}\log\big\|Dg^n|_{E^{c}_g(x)}\big\| \\
&\leq \limsup_{n\to\infty} \left( \frac{k_n(x)}{n}\log(\lambda_c) + \frac{n-k_n(x)}{n}\log(\lambda_s) \right) \\
&\leq \eta\log(\lambda_c)+(1-\eta)\log(\lambda_s).
\end{align*}
Therefore, by our choice of $\eta$ in \eqref{inq:ExpNeg}, we have that $\lambda^{cs}(g,x)$ is strictly negative $\nu$-almost every $x$.  Thus, all $\nu$-center Lyapunov exponents are strictly negative. This proves the desired result. 
\end{proof}

\begin{coro}\label{cor:negative-exponents}
Let $\eta,\mathcal O_g,\lambda_c,\lambda_s$ be as fixed in the proof of Proposition~\ref{pro:Main2}. If $\nu\in\mathrm{Prob}_{\mathrm{erg}}(g)$ satisfies $\pi_*\nu\in\mathrm{Prob}_{\mathrm{erg}}^{\hat h_0}(A)$, then all $\nu$-center Lyapunov exponents are negative.
\end{coro}

\begin{proof}
By Proposition~\ref{prop:Main1} applied to $\pi_*\nu\in\mathrm{Prob}_{\mathrm{erg}}^{\hat h_0}(A)$, $\pi_*\nu(B(\pi(p_i),r_0))<\eta/N$ for each of the $N$ centers $p_i$ of $\mathcal O_g$; as in the proof of Proposition~\ref{cor:Main1}, $B(p_i,\hat r)\subset\pi^{-1}(B(\pi(p_i),r_0))$, so
$$
\nu(\mathcal O_g)\le\sum_{i=1}^N\nu\bigl(\pi^{-1}(B(\pi(p_i),r_0))\bigr)=\sum_{i=1}^N\pi_*\nu\bigl(B(\pi(p_i),r_0)\bigr)<\eta.
$$
The remainder of the argument is verbatim the computation in the proof of Proposition~\ref{pro:Main2}: by Birkhoff's ergodic theorem, for $\nu$-a.e.\ $x$, $\lambda^{c}(g,x)\le\eta\log\lambda_c+(1-\eta)\log\lambda_s<0$.
\end{proof}

\subsection{Rigidity of the unstable rate}

Now, we present the last ingredient for the proof of Theorem \ref{teo:MainTheorem}. The rigidity of the unstable expansion is a remarkable phenomenon observed in several classes of partially hyperbolic diffeomorphisms isotopic to linear Anosov automorphisms (see \cite{CT25} for instance). In our setting, this rigidity manifests through the maximal unstable volume growth, denoted by $\chi_u(g)$ (we refer the reader to \cite{HSX08} for precise definitions). 

For a $C^1$-partially hyperbolic diffeomorphism $f$ on a compact manifold $M$, Saghin \cite[Theorem 3.1]{S14} established that if $f$ admits a closed $d_u$-form which is non-degenerate on the unstable bundle $E^u$, then its unstable volume growth is determined purely by the spectral radius of the induced action on cohomology: 
$$
\chi_u(f) = \log \operatorname{sp}(f_{*,d_u}).
$$
Furthermore, this topological property persists for any diffeomorphism $g$ that is $C^1$-close to $f$. 

In the specific case of $\mathbb{T}^d$, Proposition 3.6 in \cite{CT25} (see \cite[Proposition 3.1.2]{R14} for more details) shows that the existence of such a non-degenerate closed form is an open and closed property within the space of partially hyperbolic diffeomorphisms isotopic to $A$ that preserve the dimensions of the invariant strong bundles. 

Since the induced cohomological action is invariant under continuous deformations (isotopies), combining these two results guarantees that the unstable volume growth remains locally constant. This yields the following proposition.

\begin{propo}\label{prop:VolumeCrecimiento} 
Let $f\in \mathrm{PH}_A^c(\mathbb{T}^d)$. There exists a $C^1$-neighborhood $\mathcal{U}''$ of $f$ such that for every $g\in \mathcal{U}''$, 
$
\chi_u(g) = \log\big(\mathrm{sp}(g_{*,\,d_u})\big)
= \log\big(\mathrm{sp}(A_{*,\,d_u})\big)
= h_{\mathrm{top}}(A).
$
\end{propo}

To bound the metric entropy, we invoke a Pesin-Ruelle-type inequality due to Hua, Saghin, and Xia \cite[Theorem 3.3]{HSX08}. The power of this inequality in the partially hyperbolic setting lies in its ability to decouple the strong unstable direction from the central ones: the unstable contribution is bounded by the \textit{measure-independent} topological quantity $\chi_u(g)$, leaving only the central Lyapunov exponents as measure-dependent variables. More precisely, for any $C^1$ partially hyperbolic diffeomorphism $g$ and any $g$-invariant ergodic measure $\mu$, they showed that
$$
h_{\mu}(g) \le \chi_u(g) + \sum_{\lambda_i^{c}(g)>0}\lambda_i^{c}(g)\,m_i,
$$
where $\lambda_i^{c}(g)$ are the central Lyapunov exponents of $\mu$ and $m_i$ are their respective Oseledets multiplicities.

Next, we are ready to prove Theorem \ref{teo:MainTheorem}.

\begin{proof}[Proof of Theorem \ref{teo:MainTheorem}] 
Let 
$N \geq 1$, and let $r_1>0$ and $\varepsilon>0$ given by Proposition \ref{pro:Main2}. 
Let us consider the constants  $0 < C_2(A,\lambda_c) < 1 < C_1(A)$ given by Proposition \ref{pro:Main2}, and a $(N,r_1)$-localized deformation $g_0$ of $A$ in $ \mathcal{U}^0_{\varepsilon}(A) \cap \mathrm{PH}^c_A(\mathbb{T}^d)$ satisfying
\begin{equation*} \label{ineq:Open0}
\sup_{x \in \mathcal{O}} \|Dg_0|_{E^c_{g_0}(x)}\| < \lambda_c \quad \text{and} \quad \sup_{x \in \mathbb{T}^d \setminus \mathcal{O}} \|Dg_0|_{E^c_{g_0}(x)}\| < \lambda_s, 
\end{equation*}
where  $\lambda_c\in \left[1,\, \min\left\{C_1(A),\, \lambda_A^u\right\}\right)$ and $\lambda_s\in \left(\lambda_A^{ws},\,C_2(A,\lambda_c)\right)$. By Proposition \ref{pro:Main2} there are $h_1>0$ and a neighborhood $\mathcal{U}'$ of $g_0$ such that the $\nu$-center Lyapunov exponents of $\nu\in \mathrm{Prob}^{h_1}_{\mathrm{erg}}(g)$  are negative for every $g\in\mathcal U'$. 

Take $\mathcal{U}=\mathcal{U}'\cap\mathcal{U}''$, where $\mathcal{U}''$ is the neighborhood of $g_0$ given by Proposition \ref{prop:VolumeCrecimiento}, and let consider $g\in\mathcal{U}$. In particular, one has
\begin{equation}\label{rigidity}
    \chi_u(g) = h_{\mathrm{top}}(A), \quad\forall g \in \mathcal{U}.
\end{equation} 
We analyze any measure $\mu \in \mathrm{Prob}_{\mathrm{erg}}(g)$ in two cases:
\begin{itemize}
    \item If $h_\mu(g) \leq h_1$, then strictly $h_\mu(g) < h_{\mathrm{top}}(A)$.
    \item If $h_\mu(g) > h_1$, then $\mu \in \mathrm{Prob}_{\mathrm{erg}}^{h_1}(g)$. Since $g\in\mathcal{U}'$, the center exponents are negative. Therefore, by \eqref{rigidity} and the Hua-Saghin-Xia inequality, we obtain $$h_{\mu}(g) \le \chi_u(g) + \sum_{\lambda_i^{c}(g)>0}\lambda_i^{c}(g)\,m_i= h_{\mathrm{top}}(A).$$
\end{itemize}
In any case, we get $h_\mu(g) \le h_{\mathrm{top}}(A)$. So, by taking the supremum over all ergodic measures, via the Variational Principle we obtain
$$h_{\mathrm{top}}(g)=\sup\left\{{h}_{\nu}(g): \nu\in \text{Prob}_{erg}(g)\right\}\leq h_{\mathrm{top}}(A).$$ 
Finally, since $g$ and $A$ are semi-conjugated by $\pi$, the reverse inequality $$h_{\mathrm{top}}(A) \le h_{\mathrm{top}}(g)$$ holds. Therefore, $h_{\mathrm{top}}(g) = h_{\mathrm{top}}(A)$, concluding the proof.
\end{proof}

\section{Carrasco \textit{et al} example revisited}\label{constructionCarrasco}

{In this section, we present a three-dimensional version of the example constructed by Carrasco et al. in \cite{CLPV}, and we show that it satisfies the conclusion of Theorem \ref{teo:MainTheorem} for suitable choices of the isotopic linear part $A$. Before stating the result, we recall that a diffeomorphism $f$ is \textit{$C^1$-robustly transitive} if there exists a $C^1$-neighborhood $\mathcal{U}$ of $f$ such that every map $h \in \mathcal{U}$ is transitive.}

{The main result of this section is the following.
\begin{propo}\label{example}
Let $A\colon \mathbb{T}^3 \to \mathbb{T}^3$ be a linear Anosov diffeomorphism induced by a matrix $A\in SL(3,\mathbb{Z})$ with eigenvalues $\lambda_{ss}, \overline{\lambda_{ss}}, \lambda{_u}$ satisfying
$$0<\lambda_A^s=|{\lambda_{ss}}|\leq\frac{1}{10}<1<100\leq\lambda_u.$$ There exists a partially-hyperbolic diffeomorphism $g$ isotopic to $A$ such that
\begin{itemize}
    \item [(i)] The partially-hyperbolic splitting of $g$ is $T\mathbb{T}^3=E_g^{cs}\oplus E_g^u$, where $E_g^{cs}$ is a two-dimensional indecomposable subundle. 
    \item [(ii)] $g$ has fibers with positive topological entropy. 
    \item [(iii)] $g$ is robustly transitive. 
    \item [(iv)] The topological entropy satisfies $h_{top}(g)=h_{top}(A)$.   
\end{itemize}
\end{propo}}

For the sake of completeness, we begin by recalling the Smale horseshoe, which constitutes the main piece of the construction of the example stated in the proposition above.

\subsection{Smale's Horseshoe} Let $S = [0,1] \times [0,1]$ denote the unit square. Consider the horizontal strips $H_1, H_2 \subset S$, where $H_1$ is contained in the lower half of $S$ and $H_2$ in the upper half. Similarly, let $V_1, V_2 \subset S$ be vertical strips, with $V_1$ contained in the left half of $S$ and $V_2$ in the right half. 

We define a $C^\infty$ diffeomorphism $\sigma:\mathbb R^2 \to \mathbb R^2$ such that $\sigma(H_j)=V_j$ for $j=1,2$, $S\cap \sigma^{-1}(S)=H_1\cup H_2$, and 
$$
D\sigma(p)=\begin{pmatrix}
    \lambda_1 & 0\\
    0   & \lambda_2
\end{pmatrix},\quad \forall p\in H_1 \cup H_2,
$$
where
\begin{equation}\label{choiceec}
    \lambda_1<1/2<1<2<\lambda_2<3\quad\text{ and }\quad\lambda_1\lambda_2<1.
\end{equation}

Let $B_1$ be the semidisk at the bottom of $S$ and $B_2$ the semidisk at the top of $S$, and let $N=B_1\cup S\cup B_2$ be the topological disk that is the union of these three regions. Let $G$ be the horizontal gap between $H_1$ and $H_2$, $H_3$ be the top horizontal strip in $S$ above $H_2$, and $H_0$ be the bottom horizontal strip in $S$ below $H_1$. See Figure~\ref{fig:horseshoe-setup}.

\begin{figure}[ht]
    \centering
    \includegraphics[width=0.55\linewidth]{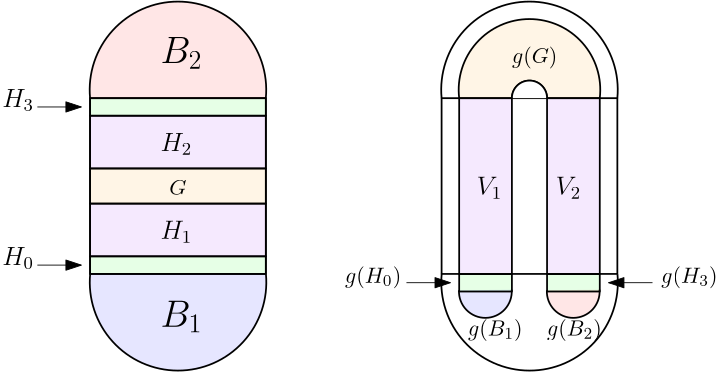}
    \caption{Geometry of the Smale horseshoe construction showing horizontal strips $H_i$, vertical strips $V_j$, and the topological disk $N = B_1 \cup S \cup B_2$.}
    \label{fig:horseshoe-setup}
\end{figure}

We extend $\sigma$ to $S$ as follows: 
\begin{itemize}
    \item $\sigma(G) \subset B_2$ and the image arcs from the top of $V_1$ to the top of $V_2$. This part of the map can be taken so that
$$
\left| \frac{\partial \sigma}{\partial x} (q) \right| = \lambda_1 \quad \text{for } q \in G.
$$
\item The image of $H_0 \cup H_3$ is contained in $B_1$; the first coordinate function of $\sigma$ on $H_0 \cup H_3$ is a contraction by a factor of $\lambda_1$, while the second coordinate function of $\sigma$ on $H_0 \cup H_3$ changes from an expansion by $\lambda_2$ on the boundary with $H_1 \cup H_2$ to a contraction at the boundary with $B_1 \cup B_2$. 
\end{itemize}
 
On the other hand, we define $\sigma$ in a such way that $\sigma(B_1) \subset B_1$, and ensure that $\sigma$ is a contraction on $B_1$, so that $\sigma$ has a unique fixed point $p_0 \in B_1$. Additionally, we define $\sigma$ on $B_2$ so that $\sigma(B_2) \subset B_1$. Thus, $\sigma$ maps the topological disk $N$ into itself. Moreover, the conditions imposed on $\sigma$ on $N$ ensure that $\sigma\vert_N$ is area-contracting.

Finally, the map is extended to all of $\mathbb R^2$ such that 
\begin{enumerate}
    \item $\sigma(B(p_0,\delta)) = N$ for some $\delta > 1$, where $B(p_0,\delta)$ denotes the ball centered in the fixed point $p_0$ with radius $\delta$, such that $N\subsetneq  B(p_0,\delta)$.
   \item For every $ x \in \mathbb{R}^2 \setminus N $, the forward orbit $ \{\sigma^n(x)\}_{n\geq 0} $ enters $ N $ in finite time and remains there. In particular, $ \lim_{n \to \infty} \text{dist}(\sigma^n(x), \Omega(\sigma)) = 0 $.
\end{enumerate}
Consequently, every point in $\mathbb{R}^2$ eventually enters the topological disk $N$ under forward iteration of $\sigma$.

\begin{lemma}\label{propo:horseshoe} Let $\sigma$ be the Horseshoe map described above. Then:
    \begin{enumerate}
        \item The nonwandering set $\Omega(\sigma)=\Lambda\cup \{p_0\}$, where 
        $$
        \Lambda=\bigcap_{j\in \mathbb Z} \sigma^j(S)=C_1\times C_2,
        $$
        where $C_1$ and $C_2$ are Cantor sets.
        \item The restriction $ \sigma|_\Lambda $ is topologically conjugate to the full shift on two symbols. Consequently, $ h_{\mathrm{top}}(\sigma) = \log 2 $.

         \item For each point $ z=(x_1,x_2) \in \Lambda $, the local unstable manifold $ W_{\mathrm{loc}}^u(z) $ is a vertical segment, while the local stable manifold $W_{\mathrm{loc}}^s(z)$ is a horizontal segment. In particular, we have
        $$
        W^u(z) = \bigcup_{n \geq 0} \sigma^n\left( \{x_1\} \times [0,1] \right) = \bigcup_{n \geq 0} \sigma^n\left( W_{\mathrm{loc}}^u(z) \right).
        $$
        
        This manifold accumulates on itself and winds densely throughout $ \Lambda $, forming a connected and indecomposable continuum. 

       \item The global unstable set of $ \Lambda $ is given by
       \begin{align*}
       W^u(\Lambda) &= \bigcup_{z\in \Lambda} W^u(z)\\
                   &=\bigcup_{j=0}^\infty \sigma^j(C_1 \times [0,1]) = \bigcap_{j=0}^\infty \sigma^j(N).
       \end{align*}
      
      Therefore, this set winds around $\Lambda$, accumulates on itself, and saturates the attractor in the unstable direction. Moreover, $ W^u(\Lambda) $ is connected, a property that follows from the transitivity of $ \Lambda $.
        \item For $\delta > 1$ satisfying $\sigma(B(p_0,\delta)) = N$, define 
        $$
        [p_0] = \{x \in \mathbb{R}^2 : \text{dist}(\sigma^n(x), p_0) < 2\delta, \ \forall n \in \mathbb{Z}\}.
        $$
        Then, $[p_0] = \overline{W^u(\Lambda)}$.
    \end{enumerate}
\end{lemma}

\begin{proof} 
We prove only item (5); for the remaining items, see \cite{Robinson}. Fix $\delta > 1$ such that $\sigma(B(p_0,\delta)) = N$. We aim to show that $[p_0] = \overline{W^u(\Lambda)}$. On one hand, by construction and item (4) one has  $p_0 \in \overline{W^u(\Lambda)} \subset N$, which implies that $\sigma^n(x) \in N$ for all $n \in \mathbb{Z}$; that is, the full orbit of $x$ remains in $N$. In particular, we have 
\begin{displaymath}
    \operatorname{dist}(\sigma^n(x), p_0) \leq \operatorname{diam}(N) \leq 2\delta, \quad \forall n \in \mathbb{Z}.
\end{displaymath}
Hence, $x \in [p_0]$.

On the other hand, given $x \in [p_0]$, the set $\mathcal{O}(x)=\overline{\{\sigma^n(x)\}_{n \in \mathbb{Z}}}$ is a nonempty, compact, and invariant subset of $\overline{B(p_0, 2\delta)}$. Since $\Lambda\cup\lbrace p_0\rbrace$ is isolated and hyperbolic it follows that $$\mathcal{O}(x)\subset\Lambda\cup\lbrace p_0\rbrace\subset \bigcup_{z\in \Lambda\cup\lbrace p_0\rbrace} W^u(z)\subset \overline{W^u(\Lambda)},$$
because $p_0\in \overline{W^u(\Lambda)}$. Therefore, if $x\neq p_0$, one has $x\in\overline{W^u(\Lambda)}$.
\qedhere
\end{proof}

\subsection{Construction of the DA map}\label{sec:DA_Carrasco} Let $\mathbb{D}_r$ denote the closed disk of radius $r$ centered at the origin in $\mathbb{R}^2$, and let $R$ be a $2\times 2$ rotation matrix whose eigenvalues $\lambda_{ss}=\alpha+\beta i$, $\overline{\lambda_{ss}}=\alpha-\beta i$ satisfy $\lambda_A^s=\vert\lambda_{ss}\vert\leq \frac{1}{10}$. Consider the horshoe map $\sigma:\mathbb{R}^2 \to \mathbb{R}^2$ given above. Following the work of Bronzi–Tahzibi \cite{BT}, we construct a one-parameter family of diffeomorphisms 
$
\{\psi_t: \mathbb{D}_2 \to \mathbb{D}_2\}_{t\in[0,1]}
$
that deforms the horseshoe into a contraction, while controlling the dynamics both inside and outside a prescribed ball. 

Assume that the family $\lbrace\psi_t\rbrace_{t\in[0,1]}$ satisfies:

\begin{enumerate}
  \item At $t=0$, the map interpolates between the horseshoe and a contraction:
  $$
  \psi_0(w)=
  \begin{cases}
  Rw, & \text{if }\|w\|\geq \tfrac{3}{2},\\
  \sigma(w), & \text{if }\|w\|\leq 1,
  \end{cases}
  \qquad \text{with } \psi_0(\mathbb{D}_{3/2}) \subset \mathbb{D}_{1}.
  $$
  \item For every $t\in[0,1]$ and $\|w\|\geq 3/2$, we fix $\psi_t(w)=Rw$.
  \item There exists $t_0\in(0,1)$ such that $\psi_{t_0}$ coincides with $\lambda_1\cdot id$  on $\mathbb{D}_1$.
  \item At $t=1$, the map reduces to a contraction: $\psi_1(w)= R(w)$ for all $w\in \mathbb{D}_2$.
  \item The derivative varies continuously with $(t,w)$:
  $$
  (t,w) \longmapsto D \psi_t(w) =
  \begin{pmatrix}
  \alpha_t(w) & \beta_t(w) \\
  \theta_t(w) & \delta_t(w)
  \end{pmatrix},
  $$
  with $\theta_t$ and $\beta_t$ uniformly bounded by $1$, and the following control conditions: Let $w=(w_1,w_2)$ and let $\lambda_1<1/2<1<2<\lambda_2$ and $\lambda_1\lambda_2<1$ and $t_0,t_1\in(0,1)$ close enough.
  \begin{enumerate} 
    \item \emph{Control in the $w_1$-direction}. For $t\in[0,t_0]$, we have $|\alpha_t|\leq \lambda_1$. while for $t\in[t_1,1]$, one has $|\alpha_t|\leq \lambda_s$.
    \item \emph{Control in the $w_2$-direction}. For $t\in[0,t_0]$, we have $1\leq|\delta_t|\leq \lambda_2$, while for $t\in[t_1,1]$, one has  $|\delta_t|\leq\lambda_A^s$.
  \end{enumerate}
\end{enumerate}

\medskip

Thus, the family $\{{\psi}_t\}_{t\in[0,1]}$ provides a  $C^1$--deformation between the Smale horseshoe and a contraction. In particular, at $t=0$ inside $\mathbb{D}_{1}$, the map ${\psi}_0$ coincides with the Smale horseshoe. For $t\geq t_0$, the map ${\psi}_t$ has a unique attracting fixed point in $\mathbb{D}_{2}$, and for all $t>t_0$ it defines a contraction. Furthermore, the family can be extended symmetrically to $t \in [-1,1]$ by setting ${\psi}_t={\psi}_{-t}$ for $t<0$. Figure 2 helps to visualize the family $\{{\psi}_t\}_{t\in[0,1]}$. 

\begin{figure}[ht]
\includegraphics[scale=0.25]{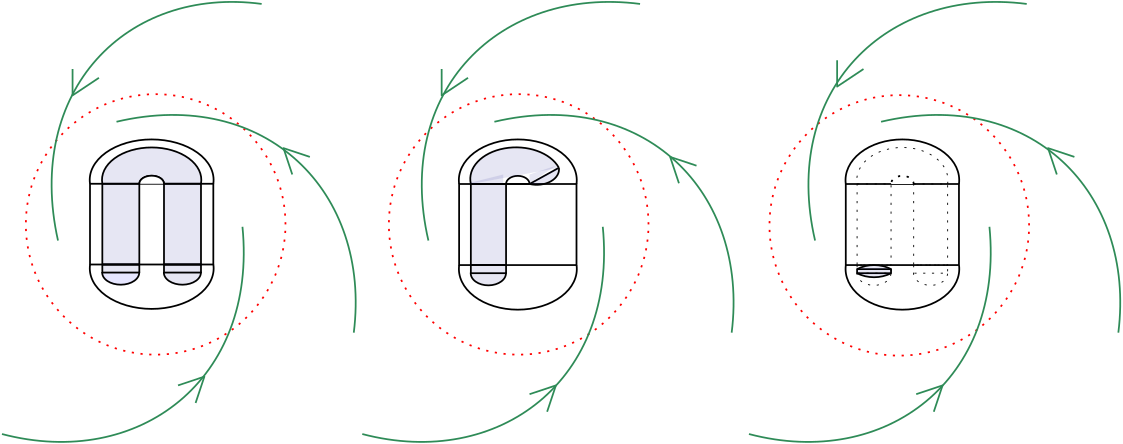}
\caption{From the left: $\psi_0$; $\psi_t,\; 0<t<t_1$; $\psi_{t_1}$.}
\end{figure}

\begin{remark}\label{specific}
An explicit form of the deformation $\psi_t$ is as follows: For instance, take $\varphi_t: \mathbb{R}^2 \to \mathbb{R}^2$, $t\in[0,1]$, such that $\varphi_t\vert_A$ is the identity map,  
$$
(\sigma\circ \varphi_t)(x, y) = 
\begin{pmatrix}
\lambda_1 & 0 \\
0 & \lambda_2+(\lambda_1-\lambda_2)\frac{t}{t_0}
\end{pmatrix}
\begin{pmatrix}
x \\
y
\end{pmatrix},\quad \forall (x,y)\in S,\quad\forall t\in[0,t_0],
$$
and $\varphi_t\vert_{B_1}$ is defined in such a way that $\varphi_0\vert_{B_1}$ is the identity map, and for every $t\in[0,t_0]$ it sends the disk $\sigma(\varphi_t(B_2))$ to the disk $\sigma(B_2)$ via rotations and translations. Then, we define $\varphi_t$, $t\in[t_0,1]$, so that 
\begin{displaymath}
   (\sigma\circ \varphi_t)(x,y)=\begin{pmatrix}
    \frac{\alpha-\lambda}{t_1-t_0}(t-t_1) & \frac{\beta}{t_1-t_0}(t-t_0)\\
    \frac{\beta}{t_0-t_1}(t-t_0)   & \frac{\alpha-\lambda}{t_1-t_0}(t-t_1)
\end{pmatrix},\quad\forall (x,y)\in\mathbb{D}_2,\quad t\in[t_0,t_1],
\end{displaymath}
and $(\sigma\circ \varphi_t)=R$ for every $t\in[t_1,1]$. Finally, define $\psi_t=\sigma\circ \varphi_t$. 
\end{remark}

Now, notice that there is a base $\mathcal{B}=\lbrace v_1,v_2, v_3\rbrace$ of $\mathbb{R}^3$ such that $A$ has the block form
$$
A=\begin{pmatrix}
 R & 0 \\
0 & \lambda_u
\end{pmatrix},
$$
where $R$ is as before. Recall that $A$ also denotes the associated linear Anosov on $\mathbb{T}^3$. The choice of the eigenvalues guarantees that it admits a dominated splitting of the form $T\mathbb{T}^3=E^s_A \oplus E^u_A$, where $E^s_A$ is a two-dimensional indecomposable subbundle. Following the approach of \cite{CLPV}, we deform the map $A$ within a small ball around a fixed point $p \in \mathbb{T}^3$ through the isotopy $\{{\psi}_t\}_{t\in[0,1]}$. {More precisely, let $r_1>0$ and $\rho>0$ small enough. Choose $0<r=r(r_1,\rho)<r_1$ so that there exists a smooth chart
$$
\phi: B(p,r/2) \to \mathbb{D}_{r_1} \times [-\rho,\rho],
\qquad \phi(p)=\mathbf{0},
$$
for which the map $A$ takes the form
$$
\phi \circ A \circ \phi^{-1}(w,z) = \big(R(w), \, \lambda_u z\big),\quad \forall (w,z)\in \mathbb{D}_{r} \times [-\rho,\rho]. 
$$}

Let $h:\mathbb{D}_{r} \to \mathbb{D}_2$  be the homothety $h(w) =\tfrac{2}{r}w$, and define the rescaled isotopy by
\begin{equation}\label{eq:deformation}
\widehat{\psi}_t = h^{-1} \circ \psi_t \circ h.
\end{equation}

By property $(b)$, for every $t\in[0,\,1]$ we have
$
\widehat{\psi}_{t}(w) =  R(w)  \ \text{for all } \|w\| \geq {3r/2}.
$
In particular, at $t=0$, $\widehat{\psi}_{t}$ exhibits a Smale horseshoe inside of $\mathbb{D}_{r/2}$.
Moreover, one has a uniform bound
 $$\sup_{w\in\mathbb{D}_{r/2}}\|D\widehat{\psi}_{t}(w)\|=\sup_{w\in \mathbb{D}_1}\|D\psi_t(w)\|<3.$$ 

We then define a perturbed map $G:\mathbb{R}^3\to \mathbb{R}^3$ by the formula 
\begin{equation}\label{eq:defpert}
G(w,z)=\begin{cases}
  \left(\widehat{\psi}_{z}(w),\lambda_u z\right), &\text{ if } \ (w,z)\in \mathbb{D}_{r}\times[-\rho,\rho]\\
(R(w),\lambda_u z),   & \text{ if } \ (w,z)\not\in \mathbb{D}_{r}\times[-\rho,\rho]
\end{cases}.
\end{equation}

Due to the smoothness from property (5) and the matching across the boundary of $\mathbb{D}_{3r/2}$, the map $G$ results in a global diffeomorphism.

Finally, we obtain the desired deformation $g:\mathbb{T}^3\to \mathbb{T}^3$ , by setting
\begin{equation} \label{def:g}
g(q)=\begin{cases}
A(q), & \text{ if } q\not\in B(p,r)\\
\phi^{-1}\circ G \circ\phi(q), & \text{ if } q\in B(p,r)
\end{cases}.
\end{equation}

We show in the following proposition that the derived from Anosov diffeomorphisms $g$ exhibit a weaker form of hyperbolicity.

\begin{lemma}\label{PH}
    The map $g$ defined in \eqref{def:g} is partially hyperbolic.
\end{lemma}

\begin{proof}
By construction of $g$, we must to exhibit a dominated splitting of $T\mathbb{T}^3=E^{cs}_g\oplus E^{u}_g$, where $E^u_g$ is uniformly expanding. Notice that the sub-bundle $E^s_A$ remains $Dg$-invariant, while $E^u_A$ loses this invariance. Consequently, we set $E^{cs}:=E^s_A$ and derive the existence of $E^u_g$ via a cone-field construction around $E^u_A$.

By recalling that $E^s_A \oplus E^u_A$ is the splitting of the linear Anosov automorphism $A$, define the unstable cone field as follows: 
$$
\mathcal{C}^u_{\alpha}(q) := \{v = v^s + v^u \in T_q\mathbb{T}^3 : \|v^s\| \leq \alpha \|v^u\| \},
$$
where $v^s \in E^s_A$, $v^u \in E^u_A$, and $\alpha > 0$ is a small constant to be determined. We will show that this cone field is forward invariant under $Dg$.

Up to the local chart $\phi$, we have that the derivative of $g$, at any point $q=(w,z) \in \mathbb{D}_{r} \times [-\rho,\rho]$, has the form
$$
Dg(w, z) =
\begin{pmatrix}
\frac{\partial \,\widehat{\psi}_{z}}{\partial w}({w})  & \frac{\partial \, \widehat{\psi}_{z}}{\partial z}({w})  \\
0 & \lambda_u
\end{pmatrix}.
$$

Furthermore, writing $\overline{w}=h(w)$, we obtain by \eqref{eq:deformation} that $D\widehat{\psi}_z(w)=D{\psi}_z(\overline{w})$. Additionally, if ${\psi}_{z}(\overline{w}):= ({\psi}_{1}(\overline{w}, z),\, {\psi}_{2}(\overline{w}, z))$, we deduce 
\begin{equation}\label{eq:DS}
\frac{\partial \,\widehat{\psi}_{z}}{\partial w}({w}) =
\begin{pmatrix}
\partial_{w_1}{\psi}_{1}(\overline{w}, z) & \partial_{w_2}{\psi}_{1}(\overline{w}, z) \\
\partial_{w_1}{\psi}_{2}(\overline{w}, z) & \partial_{w_2}{\psi}_{2}(\overline{w}, z)
\end{pmatrix}
\quad \text{and} \quad
\frac{\partial \, \widehat{\psi}_z}{\partial z}({w}) =
\begin{pmatrix}
\partial_{z}{\psi}_{1}(\overline{w},z) \\ \partial_{z}{\psi}_{2}(\overline{w}, z)
\end{pmatrix}.
  \end{equation}

Now, given every $v\in \mathcal{C}^u_{\alpha}(q)$, it follows that 
$$
Dg(w, z)(v) =
\begin{pmatrix}
\frac{\partial \, \widehat{\psi}_{z}}{\partial w}({w})  & \frac{\partial \, \widehat{\psi}_{z}}{\partial z}({w})  \\
0 & \lambda_u
\end{pmatrix}
\begin{pmatrix}
    v_1\\ v_2\\ v_3
\end{pmatrix}=:
\begin{pmatrix}
    u_1\\ u_2\\ u_3
\end{pmatrix},
$$
where $
{u}_j := \nabla \psi_{j}(\overline{w}) \cdot (v_1,v_2,v_3), \ j=1,2,
$ and ${u}_3 := \lambda_u\, {v}_3$.

On one hand, we first do the following estimations for $j=1,2$,
$$
|u_j| \leq \| (\partial_{w_1} \psi_{j}, \partial_{w_2} \psi_{j}) \|_{\mathbb{R}^2} \, \| ({v}_1,{v}_2) \|_{\mathbb{R}^2} + |\partial_{z} \psi_{j}| \, |{v}_3|.
$$
In consequence, for every $v \in C^u_{\alpha}(q)$ we deduce
\begin{align*}
\|(u_1,u_2)\|_{\mathbb{R}^2}^2 & \leq \sum_{j=1}^2 \left( 2\| (\partial_{w_1} \psi_{j}, \partial_{w_2} \psi_{j}) \|_{\mathbb{R}^2}^2 \, \| ({v}_1,{v}_2) \|_{\mathbb{R}^2}^2 + |\partial_{z} \psi_{j}|^2 \, |{v}_3|^2 \right)\\
&\leq \sum_{j=1}^2 2\left(\alpha^2\| (\partial_{w_1} \psi_{j}, \partial_{w_2} \psi_{j}) \|_{\mathbb{R}^2}^2 + |\partial_{z} \psi_{j}|^2 \right) |v_3|^2 \\
&\leq 2\left[ \alpha^2 \left\| \frac{\partial \, \widehat{\psi}_{z}}{\partial w}({w}) \right \|_F^2 + \left\| \frac{\partial \, \widehat{\psi}_{z}}{\partial z}({w})  \right\|_F^2 \right] |v_3|^2.\\
\end{align*}

On the other hand, one can assume that the map $ (w,z)\mapsto \left\| \frac{\partial \, \widehat{\psi}_{z}}{\partial z}({w})\right\|_F$ is uniformly bounded by 1. Moreover, from the relation between the Frobenius norm and the operator norm, we obtain 
$$
\left\| \frac{\partial \, \widehat{\psi}_{z}}{\partial w}({w}) \right \|_F^2
= \left\| \frac{\partial \, {\psi}_{z}}{\partial w}(\overline{w}) \right \|_F^2
\leq 2  \left\| D g (\overline{w}, z) \mid_{\mathbb{D}_1\times\{z\}} \right\|^2
\leq 18.
$$
Therefore, since $\lambda_u> 100$ (because $\lambda_A^s\leq\frac{1}{10}$ and $(\lambda_A^s)^2\lambda_u=1$), one can fix $\alpha^2\in \left[\frac{2}{\lambda_u^2-72},\, 1\right).$ 
To conclude the forward $Dg$--invariance of the unstable cone, we only need to show that
\begin{equation}\label{eq:11}
2\left[ \alpha^2 \left\| \frac{\partial \, \widehat{\psi}_{z}}{\partial w}({w}) \right \|_F^2 + \left\| \frac{\partial \, \widehat{\psi}_{z}}{\partial z}({w})  \right\|_F^2 \right]\leq \alpha^2\cdot\lambda_u^2.
\end{equation}
Indeed, using inequalities 
$$2 \left\|  \frac{\partial \, \widehat{\psi}_{z}}{\partial z}({w})  \right\|_F^2\leq 2\leq \lambda_u^2-72<\lambda_u^2-2\left\| \frac{\partial \, \widehat{\psi}_{z}}{\partial w}({w}) \right \|_F^2,$$
we obtain 
\begin{equation}\label{eq:22}
\frac{2\left\|  \frac{\partial \, \widehat{\psi}_{z}}{\partial z}({w})  \right\|_F^2}{\lambda_u^2-2\left\| \frac{\partial \, \widehat{\psi}_{z}}{\partial w}({w}) \right \|_F^2}<\frac{2}{\lambda_u^2-72}<\alpha^2.
\end{equation}
Thus, from \eqref{eq:22} one can deduce \eqref{eq:11}. So, we conclude  $$\|(u_1,u_2)\|_{\mathbb{R}^2}^2\leq \alpha^2| \lambda_u v_3|^2=\alpha^2u_3^2.$$

\smallskip

Finally, we show that every vector ${v} = (v_1,\, v_2,\, v_3) \in C^u_{\alpha}(q)$ expands in length. Indeed, we have 
\begin{align*}
\|Dg(w, z)({v})\|^2 
&= \big(\nabla \psi_{1} \cdot {v} \big)^2 
+ \big(\nabla \psi_{2} \cdot {v} \big)^2 
+ \big( \lambda_A^u\, v_3 \big)^2 \\
&\geq -\|{v}\|^2 \left( \left\| \frac{\partial \, \widehat{\psi}_{z}}{\partial w}({w}) \right \|_F^2 
+ \left\|  \frac{\partial \, \widehat{\psi}_{z}}{\partial z}({w})  \right\|_F^2 \right) 
+ \big( \lambda_u\, v_3 \big)^2.
\end{align*}

Since the vector ${v}$ satisfies $\|(v_1,\, v_2)\|^2 \leq \alpha^2 (v_3)^2$ and 
$
(v_3)^2 = \|{v}\|^2 - \|(v_1,\, v_2)\|^2
$, then 
$$
\big(\lambda_u v_3 \big)^2 
\geq \frac{(\lambda_u)^2 \|{v}\|^2}{1+\alpha^2}.
$$
Consequently,  
\begin{equation}\label{eq:IneqExp1}
\|D g(w, z)({v})\|^2 \geq \|{v}\|^2 \left( \frac{(\lambda_u)^2}{1+\alpha^2} 
- \left\| \frac{\partial \, \widehat{\psi}_{z}}{\partial w}({w}) \right \|_F^2 
- \left\|  \frac{\partial \, \widehat{\psi}_{z}}{\partial z}({w})  \right\|_F^2  \right).
\end{equation}

Therefore, if we choose $\lambda^2\in \left(1,\, \frac{\lambda_u^2-73}{2}\right]$, then  
$$
2\lambda^2 + 2\left\|  \frac{\partial \, \widehat{\psi}_{z}}{\partial z}({w})  \right\|_F^2
< \lambda_u^2 - 72 
< \lambda_u^2 - 2 \left\| \frac{\partial \, \widehat{\psi}_{z}}{\partial w}({w}) \right \|_F^2.
$$
From the above inequality, and by using that $\alpha \in (0,1)$, we obtain  
\begin{equation}\label{eq:IneqExp2}
\|D g(w, z)({v})\|^2 \geq \lambda^2 \|v\|^2.
\end{equation}
This proves the result. 
\end{proof}

\begin{lemma}\label{robustly}
    Let $g$ be DA diffeomorphism definined  by \eqref{def:g}. Then, $g$ is robustly transitive.
\end{lemma}

\begin{proof}
First, we show that $g$ is transitive. For this, notice that by construction it is area-expanding in the past. Indeed, $g\equiv A$ on $\mathbb{T}^3\setminus B(p,r_1)$, which implies that $$\det\vert Dg(x)\vert_{E^{cs}(x)}\vert=(\lambda_A^s)^2<1.$$ 
Moreover, by Remark \ref{specific}, the contraction rate on the central bundle on $B(p,r_1)$ by $Dg$ is bounded by $\lambda_1$, while the expansion rate is bounded by $\lambda_2$. Then, we have $$\det\vert Dg(x)\vert_{E^{cs}(x)}\vert\leq\lambda_1\lambda_2<1.$$ 
On the other hand, by Proposition \ref{PH} one has that $1<\lambda<\lambda_u$. This shows that $g$ expands in area any disk contained in $\mathcal{F}_g^c$ under $g^{-1}$ and $f$ enlarges the length of curves contained in $\mathcal{F}_g^u$.

Let $U, V\subset\mathbb{T}^3$ be non-empty open sets. Choose a small curve  $\gamma\subset \mathcal{F}_g^u$ contained in $U$ and a small disk  $D\subset\mathcal{F}_g^c$ contained in $V$. By the quasi-isometric property of the leaves $\mathcal{F}_g^c$ and $\mathcal{F}_g^u$, there are $m,n\in \mathbb{N}$ such that $g^n(\gamma)\cap g^{-m}(D)\neq\emptyset$, which implies $g^{m+n}(V)\cap U\neq\emptyset$. This shows that $g$ is transitive.  

Finally, by considering a small $C^1$ neighborhood $\mathcal{U}$ of $g$, all the above properties are robust. This proves the result.  
\end{proof}

Next, we present the proof of Proposition \ref{example}. 

\begin{proof}[proof of Proposition \ref{example}]
Set $d_c=\dim(E^{s}_A)=2$, $\tau=h_{top}(A)=\log(\lambda_u)$, and let $\lambda_c=\lambda_2\in(2,3)$ be as in \eqref{choiceec}. From Proposition~\ref{pro:Main2}, we have that
$$
C_1(A)=\exp\left(\frac{\log(\lambda_A^s) + \sqrt{\log(\lambda_A^s) \left[ \log(\lambda_A^s) - 2\log(\lambda_u) \right]}}{2}\right)
$$
is increasing in $\lambda_u$ and decreasing in $\lambda_A^s$ so that its value at the boundary case $\lambda_A^s=1/10$, $\lambda_u=100$ already bounds it from below for every $A$ satisfying the hypotheses of the proposition. A direct computation at this boundary case gives $C_1(A)\approx4.15>3$, so that $ \lambda_2\in[1,C_1(A))$. Moreover, the admissible interval for the contraction rate outside the perturbation ball,
$$
\left(\lambda_A^s,\ \exp\left(\frac{2\log^2(\lambda_2)}{2\log(\lambda_2)-\log(\lambda_u)}\right)\right),
$$
is nonempty at this same boundary case (its upper endpoint stays above $0.37$ throughout $\lambda_2\in(2,3)$), so we may fix $\lambda_s$ in it. Since $g\equiv A$ outside the perturbation ball, $\sup_{x\notin\mathcal{O}}\|Dg(x)|_{E^{cs}_g(x)}\|=\lambda_A^s$, which is compatible with this choice of $\lambda_s$. By Theorem~\ref{teo:MainTheorem}, there exist $r_1,\varepsilon>0$ associated to the pair $(\lambda_2,\lambda_s)$.

Let $g$ be the DA map given by \ref{def:g}, where $0<r<\min\lbrace r_1,\varepsilon\rbrace$.  Since $g$ is a DA map, it follows that it is isotopic to $A$. Moreover, by construction, it is a dynamically coherent diffeomorphism. 

By Lemma \ref{PH} we have that item $g$ is partially-hyperbolic. Moreover, 
if $\lambda=\sqrt{(\frac{\lambda_u^2-73}{2})}>3$, one has
\begin{displaymath}
    m(Dg|_{E^{u}}(y))\geq\lambda>3>\lambda_2\geq\Vert Dg|_{E^{cs}}(x)\Vert,\quad\forall x,y\in\mathbb{T}^3. 
\end{displaymath} 
This shows that 
$$m( Dg|_{E^{u}}(x))\geq\lambda\quad\text{ and }\quad \Vert Dg|_{E^{cs}}(x)\Vert< m(Dg|_{E^{u}}(y)),\quad \forall x,y\in\mathbb{T}^3.$$
Therefore, we have that $g$ is a dynamically coherent partially-hyperbolic $(1,r_1)$-localized deformation of $A$, so that $h_{top}(g)=h_{top}(A)$.

Finally, by item (5) of Lemma \ref{propo:horseshoe} it follows that $g$ has fibers with positive topological entropy, and by Lemma \ref{robustly} we have that $g$ is robustly transitive. This concludes the proof. 
\end{proof}

\section{Proof of Theorem \ref{theorem B}}  
This section is devoted to presenting the proof of Theorem \ref{theorem B}. Let $A\in SL(3,\mathbb{Z})$ with one real eigenvalue $\lambda_u>1$ and a pair of complex conjugate eigenvalues with modulus $0<\rho<1$. As in Section \ref{constructionCarrasco}, we consider a $C^\infty$ Smale horseshoe map $\sigma_0$, but in this case with more than two branches, in such a way that 
\begin{equation}\label{entropygrowth}  
    h_{\text{top}}(\sigma_0) > h_{\text{top}}(A)=\log(\lambda_u).  
\end{equation}  

If we attempted to embed this horseshoe into a small neighborhood around the fixed point of the linear map $A$ as in Section \ref{constructionCarrasco}, the resulting DA map $g$ would not be partially hyperbolic. As described in the aforementioned section, a high-entropy horseshoe requires significant expansion. Specifically, the expansion rate of $\sigma_0$ in the unstable direction, denoted by $\lambda_0^u$, is larger than $\lambda_u$. This discrepancy would violate the domination condition of the splitting $E^c_g\oplus E_g^u$. Therefore, a modification of the construction of $g$ is necessary.  

\subsection{Local perturbation of a linear Anosov automorphism}\label{subsec:LinearAnosov2}

We construct a $C^\infty$-diffeomorphism $f_1:\mathbb T^3\to\mathbb T^3$ isotopic to the linear model $A$, for which the expansion along the one-dimensional unstable direction increased in a small neighborhood of a fixed point of $A$, while its two-dimensional stable bundle, as in the case of A, does not admit a splitting into invariant one-dimensional subbundles.

Let $A\in\mathrm{SL}(3,\mathbb Z)$ be a hyperbolic matrix with a single real eigenvalue $\lambda_u>1$ and a pair of non-real complex conjugate eigenvalues:
$$
\lambda_s=\rho e^{i\theta},
\qquad
\overline{\lambda_s}=\rho e^{-i\theta},
\qquad 0<\rho<1,
$$
such that $A$ has more than one fixed point on $\mathbb T^3$.

Choose a real basis $\mathcal B=\{v_1,v_2,v_3\}$ of $\mathbb R^3$ in
which the matrix of $A$ has the block form
\begin{equation}\label{eq:A-block}
A=
\begin{pmatrix}
R&0\\
0&\lambda_u
\end{pmatrix},
\end{equation}
where $R=\rho R_\theta$ is the real $2\times2$ matrix associated with the eigenvalues $\lambda_s, \overline{\lambda_s}$ and $R_\theta$ is the rotation matrix of angle $\theta$. Providing $\mathbb R^3$, and hence $\mathbb T^3$, with the flat metric
induced by the inner product for which $\mathcal B$ is orthonormal.
Thus,
\begin{equation}\label{eq:R-conformal}
\|Rv\|=\rho\|v\|,
\qquad
\|R^{-1}v\|=\rho^{-1}\|v\|
\qquad\text{for every }v\in\mathbb R^2.
\end{equation}
From now on, we will work in these coordinates.

Let $p\in\mathbb T^3$ be a fixed point of the linear Anosov diffeomorphism induced by $A$, let $\pi:\mathbb R^3\to\mathbb T^3$ be the universal cover map such that $\pi(0)=p$, and let
$P:\mathbb R^3\to\mathbb R^3$ be the linear isomorphism whose columns are
the vectors of $\mathcal B$. Now, we define the map 
$$\Psi:\mathbb{R}^3 \to \mathbb{T}^3 \quad \text{by} \quad  \Psi:=\pi\circ P.$$
Moreover, for every $t\geq 0$, consider the family of neighborhoods of 0 given by
$$\mathcal Q_t=\mathbb D_t\times[-t,t]\subset\mathbb{R}^3.$$
Note that $\Psi(0)=p$, and for $r>0$ sufficiently
small, $\Psi$ is injective on the neighborhood 
$\overline{\mathcal Q_{4r}}$; we then set
$$
U_{4r}:=\Psi(\mathcal Q_{4r}).
$$

Now, we consider a $C^\infty$-bump function $\chi:\mathbb R^3\to[0,1]$ satisfying 
$$
\chi\equiv1\ \ \text{on}\ \ {\mathcal{Q}_1},\qquad \operatorname{supp}\chi\subset\operatorname{int}(\mathcal Q_2),
$$
where $\text{supp h}$ denotes the support of the map $h$. In particular, $\chi\equiv0$ on the complement of $\mathcal{Q}_2$. Define $C_0=\sup_{x\in\mathbb R^3}\|\nabla\chi(x)\|<\infty$. 

Fix $\lambda>\lambda_u$, to be chosen sufficiently close to $\lambda_u$, and set $$L:=\lambda-\lambda_u>0, \qquad 
B:=\begin{pmatrix} R & 0 \\ 0 & \lambda\end{pmatrix},\qquad C:=\begin{pmatrix}0 & 0 \\ 0 & L\end{pmatrix}.$$
For every $t\in [0,1]$, we define a smooth family of maps $F_t:\mathbb R^3\to\mathbb R^3$ by
\begin{equation}\label{eq:definition F}
F_t(x):=Ax+t\chi(x)\,Cx=(1-t\chi(x))Ax+t\chi(x)Bx.
\end{equation}
Writing $x=(w,z)\in\mathbb R^2\times\mathbb R$, $F_t$ has the form 
\begin{equation}\label{eq:F-fiber-form}
F_t(w,z)=\bigl(Rw,\,[\lambda_u+tL\chi(w,z)]z\bigr).
\end{equation}

By construction, $F_1$ coincides with $B$ within $\mathcal{Q}_1$, coincides with $A$ on  $\mathbb{R}^3\setminus\mathcal{Q}_2$, and interpolates between them, increasing only the unstable expansion rate.

Now, for every $r>0$, consider the homothety map $H_r: \mathbb{R}^2\times \mathbb{R} \to \mathbb{R}^2\times \mathbb{R}$ given by $H_r(w,z)=\frac{1}{2r}(w,z)$. Thus, the recaled  one-parameter family of maps $F_{t,r} : \mathbb{R}^3\to \mathbb{R}^3$ is defined by
$$F_{t,r}=H^{-1}_r\circ F_t\circ H_r,$$
where $H^{-1}_r$ denotes the inverse of $H_r$.

Finally, we define the family of deformations $f_t:\mathbb T^3\to\mathbb T^3$ by
\begin{equation}\label{eq:deformation21}
f_t(q)=\begin{cases} A(q), & q\notin U_{4r},\\ \Psi\circ F_{t,r}\circ\Psi^{-1}(q), & q\in U_{4r}. \end{cases}
\end{equation}

Every map $f_t$ is well defined because $F_{t,r}(x)=Ax$ in a neighborhood of
$\partial\mathcal Q_{4r}$ and
$\Psi(Ax)=A(\Psi(x))$. Moreover, in the global trivialization of
$T\mathbb T^3$ induced by $\mathcal B$,
\begin{equation}\label{eq:Df0-DF}
Df_t(q)=DF_t(x),\quad\forall t\in[0,1],
\end{equation}
where $q=\Psi(y)\in U_{4r}$ and $x=H_r(y)$.

Next, we see that for $t=1$ the map $f_1$ is an Anosov diffeomorphism isotopic to $A$ with an indecomposable stable direction. The proof is divided into three steps. First, we show that the local perturbation that defines $f_1$ gives rise to a $C^\infty$-diffeomorphism of $\mathbb T^3$ isotopic to $A$. Second, we establish uniform hyperbolicity using a cone-field argument. Finally, we prove that the resulting two-dimensional stable bundle does not admit a continuous invariant splitting into one-dimensional subbundles.

\begin{lemma}\label{le:Primer_Anosov}
For $\lambda>\lambda_u$ sufficiently close to $\lambda_u$, the map $f_1:\mathbb T^3\to\mathbb T^3$ defined by \eqref{eq:deformation21} is a $C^\infty$-Anosov diffeomorphism isotopic to $A$, with Anosov splitting of the form
$$
T\mathbb T^3=E^s_{f_1}\oplus E^u_{f_1},
\qquad \dim E^s_{f_1}=2,
$$
and $E^s_{f_1}$ does not admit a continuous $Df_1$-invariant splitting into one-dimensional subbundles.
\end{lemma}

\begin{proof} To begin with, let us prove that the map $f_1$ defined in \eqref{eq:deformation21} is a
$C^\infty$-diffeomorphism of $\mathbb T^3$ isotopic to $A$. For this, we first show that $f_t$ is a local diffeomorphism on $\mathbb{T}^3$ for all $t\in [0,1]$, that is, $\det Df_t(q)\neq 0$ for every $q\in\mathbb{T}^3$. 

Indeed, note that the homothety map used in the definition of $F_{t,r}$ does not change the matrix of the derivative in the normalized coordinates. Thus, inside the perturbation region, $Df_t$ is represented by $DF_t$, whereas outside that
region $D f_t=A$. Since in the normalized coordinates, for $t\in[0,1]$, one has
$$
F_t(w,z)= \bigl( Rw,\,[\lambda_u+tL\chi(w,z)]z\bigr), \ \ \text{with} \quad L=\lambda-\lambda_u,
$$
then, differentiation gives
\begin{equation}\label{eq:DF-step1}
DF_t(x)=
\begin{pmatrix}
R & 0\\
tLz\,D_w\chi(x) &
\lambda_u+tL\bigl(\chi(x)+z\,\partial_z\chi(x)\bigr)
\end{pmatrix}.
\end{equation}
Hence, the determinant of $f_t$ has the form
$$
\det Df_t(q)= \det DF_t(x)=\left[
\lambda_u+tL\bigl(\chi(x)+z\,\partial_z\chi(x)\bigr)
\right]\cdot\det (R) .
$$

On the other hand, set $C_0:=\sup_{x\in\mathbb R^3}\|\nabla\chi(x)\|<\infty$ and choose $\lambda>\lambda_u$ sufficiently close
to $\lambda_u$ so that
\begin{equation}\label{eq:L-step1}
M:=L(1+2C_0)<\lambda_u.
\end{equation}

Moreover, since $\operatorname{supp}\chi\subset\mathcal Q_2$, one has $|z|\le 2$ on the support of $\nabla\chi$. Hence, one has  
$$
\left|
\chi(x)+z\,\partial_z\chi(x)
\right|
\le
1+2C_0,$$ so that $$\ 1+2C_0+t[\chi(x)+z\partial_z\chi(x)]\geq 0,\quad\forall t\in[0,1]. 
$$
By combining the fact above with inequality \eqref{eq:L-step1}, we obtain for every $t\in[0,1]$ and every $x\in\mathbb R^3$, that
$$
\lambda_u+tL\bigl(\chi(x)+z\,\partial_z\chi(x)\bigr) \ge \lambda_u-M >0.
$$
It follows from the estimate above that $\det Df_t(q)>0$, for every $t\in[0,1]$ and all $q\in \mathbb{T}^3$. Therefore, each $f_t$ is a local diffeomorphism.

Now, the map
$$
[0,1]\times\mathbb T^3\longrightarrow\mathbb T^3,
\qquad
(t,q)\longmapsto f_t(q),
$$
is continuous, with $f_0=A$.
Hence, each $f_t$ is homotopic to $A$, and therefore
$$
\deg(f_t)=\deg(A)=1.
$$
Since $\mathbb T^3$ is compact and connected, every local diffeomorphism
$f_t:\mathbb T^3\to\mathbb T^3$ is a finite covering map. Because
$\det Df_t>0$, the number of sheets of this covering equals its degree.
Thus, each $f_t$ has exactly one sheet and is therefore a global
$C^\infty$ diffeomorphism. Consequently, $(f_t)_{t\in[0,1]}$ is an isotopy from $A$ to $f_1$.

We now prove that $f_1$ is Anosov. For
$x=(w,z)\in\mathbb R^2\times\mathbb R$, set
\begin{equation}\label{eq:a-b-step2}
a(x):=Lz\,D_w\chi(x),
\qquad
b(x):=\lambda_u+L\bigl(\chi(x)+z\,\partial_z\chi(x)\bigr).
\end{equation}

Then, by identity \eqref{eq:DF-step1}, the derivative of
$f_1$ is represented in the normalized coordinates by
$$
DF_1(x)=
\begin{pmatrix}
R & 0\\
a(x) & b(x)
\end{pmatrix}.
$$

Note that the unstable subbundle $E^u_A$ remains $Df_1$-invariant. In fact, for every $v^u\in E_A^u$, the triangular form above gives
$$
DF_1(x)(0,v^u)=(0,b(x)v^u).
$$
While outside the perturbation region, one has $Df_1=A$. Therefore, for every $q\in\mathbb T^3$
$$
Df_1(q)E_A^u(q)=E_A^u(f_1(q)).
$$
Thus, $E_A^u$ is globally $Df_1$-invariant. In contrast, the off-diagonal term $a(x)$ shows that $E_A^s$ is not, in general, $Df_1$-invariant in the perturbation region.

Before constructing a $Df_1$-invariant stable subbundle, we show that the vectors in $E^u_{f_1}:=E^u_A$ are uniformly expanded by $Df_1$.  

We first estimate some uniform bounds for the nonconstant entries $a(x)$ and $b(x)$. More precisely, we claim that
\begin{equation}\label{eq:a-b-bounds-step2}
\|a(x)\|\le M,
\qquad
|b(x)-\lambda_u|\le M,\quad\forall x\in\mathbb{R}^3. 
\end{equation}
Indeed, by definition of $C_0$, 
$$
\|a(x)\|\le2LC_0\le M,
$$
while
$$
|b(x)-\lambda_u| \le L\bigl(|\chi(x)|+|z|\,|\partial_z\chi(x)|\bigr)\le L(1+2C_0)=M.
$$

Now, we prove that the subbundle $E^u_{f_1}$ is uniformly expanded. To his end, fix 
\begin{equation}\label{eq:alpha-choice-step2}
0<\alpha<
\min\left\{
1,\sqrt{\rho^{-2}-1}
\right\}.
\end{equation}
Since $C_0$ is a constant fixed, $M\to0$ as $\lambda\to\lambda_u^+$. Thus, we may assume that
\begin{equation}\label{eq:M-choice-step2}
0<M<
\min\left\{
\lambda_u-1,\,
\frac{\alpha(\lambda_u-\rho)}{1+\alpha}
\right\}.
\end{equation}
Let $q\in \mathbb{T}^3$ and $v^u\in E_{f_1}^u(q)$. In the perturbation region, we have by inequality \eqref{eq:a-b-bounds-step2} that
$$
\|Df_1(q)v^u\|=\vert b(x)\vert\|v^u\|
\ge
(\lambda_u-M)\|v^u\|, 
$$ 
and the choice of $M$ ensures that $\lambda_u-M>1$. Since $f_1=A$ outside the perturbation region, it follows that 
$$
\|Df_1(q)v^u\|=\lambda_u\|v^u\|.
$$
Therefore, $Df_1$ expands uniformly on $E_{f_1}^u(q)$ for every $q\in\mathbb{T}^3$.  

Now, we prove the existence of a stable sub-bundle. Let us consider the stable cone field associated with the fixed splitting
$$
T\mathbb T^3=E_A^s\oplus E_A^u.
$$
For each $q\in\mathbb T^3$, define
$$
\mathcal C_\alpha^s(q)
:=
\left\{
(v^s,v^u)\in E_A^s(q)\oplus E_A^u(q):
|v^u|\le\alpha\|v^s\|
\right\}.
$$
We claim that this cone field is strictly invariant under $Df_1^{-1}$: for every $q\in\mathbb T^3$ and every nonzero vector
$$
w=(w^s,w^u)\in\mathcal C_\alpha^s(f_1(q)) \ \implies \ [Df_1(q)]^{-1}w
\in
\operatorname{int}\mathcal C_\alpha^s(q).
$$

We set 
$(\widetilde{w}^s,\,\widetilde{w}^u):=[Df_1(q)]^{-1}(w)$. Since the inverse matrix of $Df_1(q)$ is given by $$[Df_1(q)]^{-1}=\begin{pmatrix} R^{-1} & 0 \\ -b^{-1}aR^{-1} & b^{-1}\end{pmatrix},$$
then, 
$$
\widetilde{w}^s=R^{-1}(w^s),\qquad \widetilde{w}^u=-b^{-1}aR^{-1}(w^s) + b^{-1}w^u.
$$
In particular, $$[Df_1(q)]^{-1}w \in \operatorname{int}\mathcal C_\alpha^s(q)\Leftrightarrow\frac{|\widetilde{w}^u|}{\|\widetilde{w}^s\|}\leq\alpha.$$

A straightforward computation yields $\|\widetilde{w}^s\|=\|R^{-1}(w^s)\|=\rho^{-1}\|w^s\|$. Thus, since $|w^u|\le\alpha\|w^s\|$,
$$
|\widetilde{w}^u|\le |b|^{-1}\big(\|a\|\,\|R^{-1}(w^s)\| + |w^u|\big)
\le |b|^{-1}\big(\|a\|\rho^{-1}\|w^s\| + \alpha\|w^s\|\big).
$$
Therefore,
 \begin{equation}\label{eq:Deformation2_2}
\frac{|\widetilde{w}^u|}{\|\widetilde{w}^s\|}\le \frac{|b|^{-1}\|w^s\|\big(\|a\|\rho^{-1} + \alpha)}{\rho^{-1}\|w^s\|}=\frac{\|a\| + \alpha\rho}{|b|}.
\end{equation}
Then, by inequality \eqref{eq:a-b-bounds-step2} and the choice of $M$ in \eqref{eq:M-choice-step2}, we have that 
\begin{equation*}
\frac{|\widetilde{w}^u|}{\|\widetilde{w}^s\|}\le \frac{M+\alpha\,\rho}{\lambda_u-M}<\alpha.
\end{equation*}
This shows that the backward invariance of the stable cone field $C^s_{\alpha}$. 

We next prove uniform contraction of the longitude of vectors on the stable cone by $Df_1$. Let $q \in \mathbb{T}^3$ and $w = (w^s, w^u) \in C_\alpha^s(q)$. We give the following estimate:
\begin{align*}
\|[Df_1(q)]^{-1} w\|^2 &= \|\tilde{w}^s\|^2 + |\tilde{w}^u|^2 \\
&\ge \|\tilde{w}^s\|^2.
\end{align*}
Since $w \in C_\alpha^s(q)$, $$\|w\|^2 = \|w^s\|^2 + \|w^u\|^2 \le (1 + \alpha^2) \|w^s\|^2.$$ Additionally, $\|\tilde{w}^s\|^2 = \rho^{-2} \|w^s\|^2$. Then, by combining inequalities above, it follows that:
$$
\|[Df_1(q)]^{-1} w\|^2 \ge \frac{\rho^{-2}}{1 + \alpha^2} \|w\|^2 \implies \|[Df_1(q)]^{-1} w\| \ge \frac{1}{\rho \sqrt{1 + \alpha^2}} \|w\|
$$
So, by the choice of $\alpha$ in \eqref{eq:alpha-choice-step2}, we deduce that
$$
\frac{1}{\rho \sqrt{1 + \alpha^2}} > 1.
$$
Therefore, every vector in $C_\alpha^s$ is uniformly contracted under $Df_1$.

In this way, by the cone criterion applied to $f_1^{-1}$, there exists a continuous
two-dimensional $Df_1$-invariant subbundle $E_{f_1}^s$ satisfying
$$
E_{f_1}^s(q)
\subset
\mathcal C_\alpha^s(q),
\qquad
\forall q\in\mathbb T^3.
$$
Moreover, by construction, $E^u_{f_1}(q):=E_A^u(q)$ is a one-dimensional subspace transverse to $ C_\alpha^s(q)$. In particular, for every $q\in\mathbb T^3$ one has
$$
E_{f_1}^s(q)\cap E_{f_1}^u(q)=\{0\}.
$$
Thus, we obtain a continuous $Df_1$-invariant splitting
$$
T\mathbb T^3=E_{f_1}^s\oplus E^u_{f_1},
$$ 
where $E_{f_1}^s$ is a two-dimensional sub-bundle uniformly contracted and
$E_{f_1}^{u}$ is a one-dimensional sub-bundle uniformly expanded. Therefore, $f_1$ is an Anosov diffeomorphism.

Finally, we prove that $E^s_{f_1}$ does not admit a splitting into one-dimensional subbundles. For this, is enough to show that $E^s_{f_1}(p)$ is precisely $E^s_A(p)$.  Indeed, recall that the local perturbation fixes the origin, that is, $f_1(p)=p$. Moreover, $\chi\equiv1$ in a neighborhood of the origin, hence
$D_w\chi(0)=0$ and $\partial_z\chi(0)=0$. Thus, the derivative at $p$ is
$$
Df_1(p)=
\begin{pmatrix}
R & 0\\
0 & \lambda
\end{pmatrix}.
$$
Hence, $$
Df_1(p)|_{E_{A}^s(p)}=R=\rho R_\theta.
$$
Thus, since the cone field $\mathcal{C}_{\alpha}^s$ is $Df_1^{-1}$-invariant, we have that both two-dimensional spaces coincide, i.e.,
$$
E_{f_1}^s(p)=E_A^s(p).
$$
In particular,
$$
Df_1(p)|_{E_{f_1}^s(p)}=R,
$$
 and by the assumption that $R$ is a matrix generated by a pair of non-real complex eigenvalues, then $Df_1(p)|_{E_{f_1}^s(p)}$ has no nontrivial real one-dimensional invariant subspaces.  
\end{proof}

\begin{remark}
We have the following remaerks: 
\begin{itemize}
    \item By construction of $f_1$, the derivative $Df_1$ is lower block-triangular with diagonal blocks $R$ and $b(x)$, and $R$ is exactly the linear map $A|_{E^s_A}$. However, $E^s_A$ itself is \emph{not} $Df_1$-invariant on $B(p,4r)$ because $a(x)$ is generally nonzero there. So, the Anosov splitting $E^s_{f_1}\oplus E^u_{f_1}$ obtained from Lemma~\ref{le:Primer_Anosov} is a nonlinear perturbation of $E^s_A\oplus E^u_A$, tangent to $E^s_A$ at $p$. 
    \item By construction of $f_1$, one has the stable leaf $\mathcal{W}_{f_1}^s(p)$ coincides with that of $A$, i.e., $\mathcal{W}_{f_1}^s(p)=\mathcal{W}_{A}^s$. 
\end{itemize}
\end{remark}

In this way, since $f_1$ and $A$ are isotopic, by classical results of Franks and Manning (see \cite{Franks70,Manning74}) on Anosov diffeomorphisms on the torus and Lemma~\ref{le:Primer_Anosov} we obtain the following fact:
\begin{lemma}
The Anosov diffeomorphism $f_1:\mathbb T^3\to\mathbb T^3$ defined in \eqref{eq:deformation21} belongs to the isotopy class of $A$. In particular, the topological entropies are preserved, that is, $h_{\mathrm{top}}(f_1)=h_{\mathrm{top}}(A)$.
\end{lemma}


Now, we construct a derived-from-Anosov diffeomorphism $f$ satisfying the conclusions of Theorem \ref{theorem B}, that is, with topological entropy strictly larger than that of $f_1$, and so strictly larger than that of $A$. The construction is inspired by the example introduced in~\cite{CLPV} and follows the same general strategy developed in Section~\ref{sec:DA_Carrasco}.

More precisely, starting from the fixed point $p$ of $f_1$, we perform an isotopic deformation of $f_1$ inside a sufficiently small open region $U_{r_1}\subset U_{4r}$, where $0<r_1<4r$. The resulting diffeomorphism $f$ satisfies the following properties:
 
\begin{enumerate}
    \item[{[M1]}] $f$ is partially hyperbolic, admitting a splitting $T\mathbb{T}^3=E^{cs}_{f}\oplus E^u_{f}$, where $\text{dim}(E^u_f)=1$ and the two-dimensional invariant sub-bundle $E^{cs}_{f}$ does not admit a splitting into one-dimensional sub-bundles.
    \item [{[M2]}] The topological entropy satisfies $$h_{\text{top}}(f)>h_{\text{top}}(A).$$
     \item [{[M3]}]  {$f$ is robustly transitive.} 
\end{enumerate}

We proceed to describe the construction of $f$. Let $A: \mathbb{T}^3\to \mathbb{T}^3$ be the linear Anosov diffeomorphism introduced in Section \ref{subsec:LinearAnosov2}. Choose a $C^{\infty}$-horseshoe map $\sigma_0:\mathbb{R}^2\to \mathbb{R}^2$ such that $h_{\text{top}}(\sigma_0)>\log(\lambda_u)$. We denote by $H_0\subset \mathbb{R}^2$ the hyperbolic Cantor set associated with $\sigma_0$ . In this case, $$h_{\text{top}}(\sigma_0)=h_{\text{top}}\left(\sigma_0\mid_{H_0}\right).$$ Take  $\lambda>\lambda_u$ such that $\log(\lambda_u)<h_{\text{top}}(\sigma_0)<\log(\lambda)$. 

Let $4r>0$ be the radius of the support of the deformation defining $f_1$ in \eqref{eq:deformation21}. Fix $r_1 \in (0,\, 4r)$ and consider a one-parameter family of diffeomorphisms $\psi_t: \mathbb{R}^{2}\to \mathbb{R}^{2}$, $t\in[0,\,1]$ such that
\begin{enumerate}
\item  $\psi_0\vert_{\mathbb{D}_r}$ exhibits a Smale horseshoe with topological entropy $h_{\text{top}}(\sigma_0)$.
\item for every $t\in [0,\,1]$, $\psi_t(0)=0$ and $\psi_t({\mathbb{D}_{2r}})$ is contained in ${\mathbb{D}_{r}}$,
\item for every $t\in [0,\,1]$, $\psi_t$ coincides with  $F_1(\cdot, 0)$ on ${\mathbb{D}_{2r}}\setminus \mathbb{D}_{r}$,
\item $\psi_1$ coincides with $F_1$,
\end{enumerate}

As in Section \ref{sec:DA_Carrasco} and identity \eqref{eq:defpert}, we construct a diffeomorphism $G:\mathbb{R}^3\to \mathbb{R}^3$ that extends $\psi_0$ and coincides with $F_1$ outside $\mathcal Q_{2r}$. This defines a diffeomorphism $f:\mathbb T^3\to\mathbb T^3$ by
\begin{equation} \label{def:g_0}
f(q)=\begin{cases} A(q), & q\notin U_{4r},\\ \Psi\circ G\circ\Psi^{-1}(q), & q\in U_{4r}, \end{cases}
\end{equation}
where $\Psi$ is the chart considered in the construction of $f_1$. Since $G=F_1$ outside $\mathcal Q_{2r}\subset\mathcal Q_{4r}$, $f$ matches $A$ smoothly across $\partial U_{4r}$, as in \eqref{eq:deformation21}.

Next, we proceed to prove Theorem \ref{theorem B}. 

\begin{proof}[proof of Theorem \ref{theorem B}]
Let $f:\mathbb{T}^3\to\mathbb{T}^3$ given in \eqref{def:g_0}. By construction, it is clear that $f$ is isotopic to $f_1$, and so isotopic to $A$. The partial hyperbolicity of $f$ follows the same argument as in Lemma \ref{PH}. 

Since $A$ has more than one fixed point (Section~\ref{subsec:LinearAnosov2}) and the perturbation region defining $f$ can be taken arbitrarily small, some fixed point $p'\ne p$ of $A$ lies outside of $U_{4r}$, so $f\equiv A$ on a neighborhood of $p'$; in particular, $f(p')=A(p')=p'$, $E^{cs}_f(p')=E^s_A$, and $Df(p')|_{E^{cs}_f(p')}=R$, which has no real invariant lines. This shows property [M1].

Since the definition of $\sigma_0$ is analogous to that given in Section~\ref{sec:DA_Carrasco}, the robust transitivity of $f$ follows from the same argument given in Proposition \ref{robustly}. So, [M3] follows. 

Finally, we show that $h_{\text{top}}(f)>h_{\text{top}}(A)$. Let $\widehat H_0\subset \mathbb{T}^3$ denote the image of $H_0$ under the chart used to define $f$, so that $\widehat H_0$ is a compact $f$-invariant set and $f|_{\widehat H_0}$ is conjugate to $\sigma_0|_{H_0}$. Thus, by monotonicity of the topological entropy we have
$$h_{\text{top}}(f)\geq h_{\text{top}}\bigl(f|_{\widehat H_0}\bigr)=h_{\text{top}}(\sigma_0|_{H_0})=h_{\text{top}}(\sigma_0)>\log(\lambda_u)=h_{\text{top}}(A).$$ Therefore, [M2] is obtained. This concludes the proof.
\end{proof}

\section{Proof of Theorem~\ref{coroequilibrium} and Theorem~\ref{theorem D}}

The main step in the proof of Theorem~\ref{coroequilibrium} and Theorem~\ref{theorem D} is a measure-theoretic injectivity property of the semiconjugacy $\pi$ between $g$ and $A$, in the spirit of~\cite[\S6.3]{BF}: $\pi$ is injective at $\nu$-almost every point, for every ergodic $\nu\in\mathrm{Prob}(g)$ whose projection $\pi_*\nu$ has entropy larger than $h_0$. Let us consider the measurable sets
$$
M':=\{x\in \mathbb{T}^d:\ \pi^{-1}(\pi(x))=\{x\}\}\quad\textit{ and }\quad M'':=\pi(M').
$$

\begin{lemma}\label{lem:pi-injective}
Let $\eta, h_0$ and $\mathcal O_g$ be as in the proof of Proposition~\ref{pro:Main2}, via Proposition~\ref{cor:Main1}. If $\nu\in\mathrm{Prob}_{\mathrm{erg}}(g)$ satisfies $\pi_*\nu\in\mathrm{Prob}_{\mathrm{erg}}^{h_0}(A)$, then $\nu(M')=1$.
\end{lemma}

\begin{proof}
First, by the proof of Lemma \ref{le:keylemma} one has  $$\pi^{-1}(\pi(x))\subset\mathcal W^c_g(x),\quad x\in\mathbb T^d.$$ 

Now, since $g\in\mathrm{PH}^c_A(\mathbb T^d)$ is dynamically coherent~\cite{FPS14}, the leaves of $\mathcal W^c_g$ are $C^1$ and uniformly embedded on $\mathbb T^d$. Thus, there exist constants \(0<\rho_1<\rho_0\), uniform over \(g\in\mathcal U\), such that if $x,y$ lie on the same central leaf and $d(x,y)<\rho_1$, then $d^{c}(x,y)\le\rho_0$, where $d^{c}$ is the intrinsic distance along the central leaf. Choose \(0<\delta<\rho_1/2\) in the application of Lemma~\ref{le:TopEstab} producing \(\pi\). Notice that this only requires shrinking $\varepsilon$, without affecting $\lambda_c,\lambda_s,\eta$.

Since $\pi_*\nu\in\mathrm{Prob}_{\mathrm{erg}}^{ h_0}(A)$, by Proposition~\ref{prop:Main1} one obtains $\nu(B(p_i, r))<\eta/N$ for each of the $N$ centers $p_i$ of $\mathcal O_g$. Hence,
$$
\nu(\mathcal O_g)\le\sum_{i=1}^N\nu\bigl(B(p_i, r)\bigr)<\eta.
$$
On the other hand, since $E^{c}_g$ is $Dg$-invariant, for $z=g(w)$ one has
\begin{equation}\label{eq:conorm-dual}
m\bigl(Dg^{-1}(z)|_{E^{c}_g(z)}\bigr)=\bigl\|Dg(w)|_{E^{c}_g(w)}\bigr\|^{-1}.
\end{equation}
By hypothesis $\|Dg(w)|_{E^{c}_g(w)}\|<\lambda_c$ if $w\in\mathcal O_g$, and $<\lambda_s$ if $w\notin\mathcal O_g$; hence \eqref{eq:conorm-dual} gives $m(Dg^{-1}(z)|_{E^{c}_g(z)})>\lambda_c^{-1}$ and $>\lambda_s^{-1}>1$, respectively.

Let $x$ be a $\nu$-typical point and suppose $y\in\pi^{-1}(\pi(x))$, with $y\ne x$. Then, $y\in\mathcal W^{c}_g(x)$, and since $d(x,y)<2\delta<\rho_1$, we have by the choice of $\delta$ that $d^{c}(x,y)\le\rho_0$. Set $z_j:=g^{-j}(x)$, $w_j:=g^{-j}(y)$ for $j\ge0$. As long as $d^{c}(z_j,w_j)\le\rho_0$ for $j=0,\dots,n-1$, applying the fundamental theorem of calculus along a minimizing arc in the leaf together with \eqref{eq:conorm-dual} yields
$$
d^{c}(z_n,w_n)\;\ge\;\Bigl(\prod_{j=1}^{n} c(z_j)\Bigr)\,d^{c}(x,y),
\qquad
c(z)=\begin{cases}\lambda_c^{-1}, & z\in\mathcal O_g,\\ \lambda_s^{-1}, & z\notin\mathcal O_g.\end{cases}
$$
Since $\nu(\mathcal O_g)<\eta$, and $\eta$ was chosen in Lemma~\ref{le:Aux} so that $\eta\log\lambda_c+(1-\eta)\log\lambda_s<0$, and $t\mapsto t\log\lambda_c+(1-t)\log\lambda_s$ is increasing, the condition $\nu(\mathcal O_g)<\eta$ already implies
$$
\kappa:=-\bigl[\nu(\mathcal O_g)\log\lambda_c+(1-\nu(\mathcal O_g))\log\lambda_s\bigr]>0.
$$
By Birkhoff's ergodic theorem applied to $g^{-1}$, one has for $\nu$-a.e.\ $x$ that $$\frac1n\sum_{j=1}^n\log c(z_j)\to\kappa>0.$$ Since $d^{c}(x,y)>0$, the product $\prod_{j=1}^n c(z_j)$ diverges, so there exists a minimum index $n_0$ such that $d^{c}(z_{n_0},w_{n_0})>\rho_0$, and the estimate above is valid up to this step by construction. By the choice of \(\rho_0,\rho_1\), this implies \(d(z_{n_0},w_{n_0})>\rho_1\). But $\pi(z_{n_0})=A^{-n_0}\pi(x)=A^{-n_0}\pi(y)=\pi(w_{n_0})$, so that $d(z_{n_0},w_{n_0})<2\delta<\rho_1$, which is a contradiction. Therefore, $y=x$ for every $y\in\pi^{-1}(\pi(x))$ and $\nu$-a.e.\ $x$. Hence, $\nu(M')=1$, proving the result.
\end{proof}

\begin{remark}
The proof of Lemma~\ref{lem:pi-injective} adapts the entropy-conjugacy argument of Buzzi and Fisher~\cite[\S6.3]{BF}. In their setting, they assume a dominated splitting rather than partial hyperbolicity, whose deformations are $\gamma$-nearly hyperbolic (Def.~2.4 in~\cite{BF}) with $\gamma$ close to $0$, i.e., the central expansion/contraction rate must stay close to $1$. In the presence of partial hyperbolicity, the central direction may expand by an arbitrarily large factor $\lambda_c>1$ inside the deformation region, provided it is compensated by a sufficiently strong contraction $\lambda_s<1$ outside it, as in Lemma~\ref{le:Aux}.
\end{remark}

\begin{coro}\label{cor:bijection}
Let $\mu\in\mathrm{Prob}_{\mathrm{erg}}^{h_0}(A)$. There exists $\nu\in\mathrm{Prob}_{\mathrm{erg}}(g)$ with $\pi_*\nu=\mu$, and it is the unique ergodic measure of $g$ with this property: if $\nu_1,\nu_2\in\mathrm{Prob}_{\mathrm{erg}}(g)$ satisfy $\pi_*\nu_i=\mu$, then $\nu_1=\nu_2$. Moreover $h_\nu(g)=h_\mu(A)$, and for every continuous $\theta:\mathbb{T}^d\to\mathbb R$,
$${}
h_\nu(g)+\int\theta\circ\pi\,d\nu \;=\; h_\mu(A)+\int\theta\,d\mu.
$$
Furthermore, $\mu(M'')=1$.
\end{coro}

\begin{proof} Notice that, since $\pi_*$ is onto, there exists $\nu_0\in\mathrm{Prob}(g)$ such that $\pi_*\nu_0=\mu$. Consider the ergodic decomposition of $\nu_0$, taht is $\nu_0=\int\nu_\alpha\,d\tau(\alpha)$. Then
$$
\mu=\pi_*\nu_0=\int\pi_*\nu_\alpha\,d\tau(\alpha).
$$
Since $\mu$ is extremal in $\mathrm{Prob}(A)$, it follows that $\pi_*\nu_\alpha=\mu$ for $\tau$-almost every $\alpha$. By Lemma~\ref{lem:pi-injective}, we then have $\nu_\alpha(M')=1$ for $\tau$-almost every $\alpha$. Since $\pi$ is injective on $M'$, all such ergodic components coincide. Hence $\nu_0$ is ergodic.

For uniqueness, let $\nu_1,\nu_2\in\mathrm{Prob}_{\mathrm{erg}}(g)$ satisfy $\pi_*\nu_i=\mu$. By Lemma~\ref{lem:pi-injective}, one has $\nu_1(M')=\nu_2(M')=1$. Since $\pi$ is injective on $M'$, we obtain $$\nu_1=(\pi|_{M'})^{-1}_*\mu=\nu_2.$$

For the entropy identity, fix $\nu$ as above. By Lemma~\ref{lem:pi-injective}, $\nu(M')=1$ and $\pi^{-1}(\pi(x))=\lbrace x\rbrace$, for every $x\in M'$. Then, the Ledrappier--Walters inequality,
$$
h_{\mu}(A)\le h_\nu(g)\le h_{\mu}(A)+\int_M h_{\mathrm{top}}\bigl(g,\pi^{-1}(\pi(x))\bigr)\,d\nu(x)=h_{\mu}(A).
$$
The identity for $\theta$ follows from this and the fact that $$\int\theta\circ\pi\,d\nu=\int\theta\,d(\pi_*\nu)=\int\theta\,d\mu.$$

Finally, since $\nu(M')=1$ and $M'\subset\pi^{-1}(\pi(M'))=\pi^{-1}(M'')$, one has
$$
1\geq\mu(M'')=\pi_*\nu(M'')=\nu\bigl(\pi^{-1}(M'')\bigr)\ge\nu(M')=1.
$$
This proves the result.
\end{proof}

Let $\psi:M\to\mathbb{R}$ be a continuous potential for which there exists a unique equilibrium state for the Anosov $A$. Let $d_c:=\dim(E^{c}_g)$ as in Theorem~\ref{theorem D}. Assume that
\begin{equation}\label{smalloscilation}
\Delta_\psi:=\sup_M\psi-\inf_M\psi \;<\; h_{\mathrm{top}}(A)-h_0 \;-\;d_c\log(\lambda_c).
\end{equation}

\begin{lemma}\label{hentopy}
Let $\mu_\psi$ be the unique equilibrium state for $(A,\psi)$, where $\psi$ satisfies \eqref{smalloscilation}. Then, $\mu_\psi\in\mathrm{Prob}_{\mathrm{erg}}^{h_0}(A)$.
\end{lemma}

\begin{proof}
Suppose $h_{\mu_\psi}(A)\le h_0$. Then,
$$
P(\mu_\psi,\psi)=h_{\mu_\psi}(A)+\int\psi\,d\mu_\psi \le h_0 + \sup_{M}\psi.
$$
Since $A$ is a linear Anosov automorphism, the Haar measure $m$ satisfies $h_m(A)=h_{\mathrm{top}}(A)$, so
$$
P(m,\psi)=h_{\mathrm{top}}(A) + \int\psi\,dm \ge h_{\mathrm{top}}(A)+\inf_{M}\psi.
$$
Therefore, by \eqref{smalloscilation},
$$
P(m,\psi)-P(\mu_\psi,\psi)
\ge \bigl[h_{\mathrm{top}}(A)-h_0\bigr] - \Delta_\psi \;\ge\; d_c\log(\lambda_c)\;>0,
$$
contradicting that $\mu_\psi$ attains $P_{\mathrm{top}}(A,\psi)$. Hence $h_{\mu_\psi}(A)>h_0$.
\end{proof}

 Define $G:=[h_{\mathrm{top}}(A)-h_0]-\Delta_\psi>0$. Notice that $d_c\log(\lambda_c)<G$ by \eqref{smalloscilation}. By the same computation as in the proof of Lemma~\ref{hentopy}, we have $$P_\mu(A,\psi)\le P_A(\psi)-G$$ for every $\mu\in\mathrm{Prob}_{\mathrm{erg}}(A)$ with $h_\mu(A)\le h_0$.

Define $\varphi:=\psi\circ\pi$. By Lemma~\ref{hentopy}, $\mu_\psi\in\mathrm{Prob}_{\mathrm{erg}}^{h_0}(A)$. So, by Corollary~\ref{cor:bijection}, there exists a unique $\nu^+\in\mathrm{Prob}_{\mathrm{erg}}(g)$ with $\pi_*\nu^+=\mu_\psi$, and $h_{\nu^+}(g)=h_{\mu_\psi}(A)$. Since $\mu_\psi\in\mathrm{Prob}_{\mathrm{erg}}^{h_0}(A)$ and $\pi_*\nu^+=\mu_\psi$, Corollary~\ref{cor:negative-exponents} shows that all $\nu^+$-center Lyapunov exponents are negative. Therefore, $\nu^+$ is hyperbolic.

Now, we present the proof of Theorem \ref{coroequilibrium}. 

\begin{proof}[Proof of Theorem \ref{coroequilibrium}]
By Corollary~\ref{cor:bijection} applied to $\mu_\psi$, with $\theta=\psi$,
$$
h_{\nu^+}(g)+\int\varphi\,d\nu^+ = h_{\mu_\psi}(A)+\int\psi\,d\mu_\psi = P_A(\psi),
$$
so that $P_g(\varphi)\ge P_A(\psi)$.

\medskip
\noindent\textbf{Claim.} \textit{For every ergodic $\nu'\in\mathrm{Prob}(g)$, with $\nu'\ne\nu^+$,
$$
h_{\nu'}(g)+\int\varphi\,d\nu' \;<\; P_A(\psi).
$$}

Let $\mu':=\pi_*\nu'$. We have the following cases: 

\smallskip
\noindent\emph{Case 1.} If $h_{\mu'}(A)>h_0$, then $\mu'\in\mathrm{Prob}_{\mathrm{erg}}^{h_0}(A)$. Hence, by Corollary~\ref{cor:bijection}, $\nu'$ is the unique ergodic measure of $g$ projecting to $\mu'$, and $h_{\nu'}(g)=h_{\mu'}(A)$. Then,
$$
h_{\nu'}(g)+\int\varphi\,d\nu' = h_{\mu'}(A)+\int\psi\,d\mu' = P_{\mu'}(A,\psi).
$$
If $\mu'=\mu_\psi$, the uniqueness property given in Corollary~\ref{cor:bijection} gives $\nu'=\nu^+$, contrary to assumption. Thus, $\mu'\ne\mu_\psi$ and  $P_{\mu'}(A,\psi)<P_A(\psi)$, since $\mu_\psi$ is the unique equilibrium state of $(A,\psi)$.

\smallskip
\noindent\emph{Case 2.} If $h_{\mu'}(A)\le h_0$, Lemma~\ref{le:keylemma} gives, unconditionally on the entropy of $\nu'$, $h_{\nu'}(g)\le h_{\mu'}(A)+d_c\log(\lambda_c)$. Hence,
$$
h_{\nu'}(g)+\int\varphi\,d\nu' \;\le\; P_{\mu'}(A,\psi)+d_c\log(\lambda_c) \;\le\; P_A(\psi)-G+d_c\log(\lambda_c) \;<\; P_A(\psi),
$$
because $d_c\log(\lambda_c)<G$.

\smallskip
This proves the Claim. Together with $P_g(\varphi)\ge P_A(\psi)$, it shows $P_g(\varphi)=P_A(\psi)$, attained at $\nu^+$. Thus, $\nu^+$ is an equilibrium state for $(g,\varphi)$.

Now, we prove the uniqueness of the equilibrium state. For this end, let $\tilde\nu$ be any equilibrium state of $(g,\varphi)$, decomposed into ergodic components $\tilde\nu=\int\nu'\,d\tau(\nu')$. By affinity of the pressure functional and the above Claim we have
$$
\int\Bigl(h_{\nu'}(g)+\int\varphi\,d\nu'\Bigr)\,d\tau(\nu') = h_{\tilde\nu}(g)+\int\varphi\,d\tilde\nu = P_A(\psi).
$$
Notice that, by the Claim above, the integrand given in the left side of the inequality above is less or equal than $P_A(\psi)$, with equality only at $\nu'=\nu^+$. Hence, $\tau$ is concentrated at $\nu^+$, so that $\tilde\nu=\nu^+$.
\end{proof}

Finally, we present the proof of Theorem~\ref{theorem D}. 

\begin{proof}[Proof of Theorem~\ref{theorem D}]
Since $h_\nu(g)>h_1=h_0+d_c\log(\lambda_c)$, Lemma~\ref{le:keylemma} gives $h_\mu(A)>h_0$, that is, $\mu\in\mathrm{Prob}_{\mathrm{erg}}^{h_0}(A)$.

Now, let $\psi:=\psi_\mu$ be the potential given by Phelps~\cite{P} for which $\mu$ is the unique equilibrium state of $(A,\psi)$, and set $\varphi:=\psi\circ\pi$. By Corollary~\ref{cor:bijection}, $\nu$ is the unique ergodic measure of $g$ projecting to $\mu$, and $h_\nu(g)=h_\mu(A)$. So, 
$$
h_\nu(g)+\int\varphi\,d\nu = h_\mu(A)+\int\psi\,d\mu = P_A(\psi),
$$
so $P_g(\varphi)\ge P_A(\psi)$.

\medskip
\noindent\textbf{Claim.} \textit{For every ergodic $\nu'\in\mathrm{Prob}(g)$, $\nu'\ne\nu$,
$$
h_{\nu'}(g)+\int\varphi\,d\nu' \;<\; P_A(\psi).
$$}

Let $\mu':=\pi_*\nu'$. We have the following cases: 

\smallskip
\noindent\emph{Case 1.} If $h_{\mu'}(A)>h_0$, then as in the proof of Theorem~\ref{coroequilibrium}, Corollary~\ref{cor:bijection} gives $h_{\nu'}(g)+\int\varphi\,d\nu' = P_{\mu'}(A,\psi)$. If $\mu'=\mu$, uniqueness gives $\nu'=\nu$, contrary to assumption. Hence, $\mu'\ne\mu$, and $P_{\mu'}(A,\psi)<P_A(\psi)$, because $\mu$ is the unique equilibrium state of $(A,\psi)$.

\smallskip
\noindent\emph{Case 2.} If $h_{\mu'}(A)\le h_0$, Lemma~\ref{le:keylemma} gives $h_{\nu'}(g)\le h_{\mu'}(A)+d_c\log(\lambda_c)$, so that
$$
h_{\nu'}(g)+\int\varphi\,d\nu' \;\le\; P_{\mu'}(A,\psi)+d_c\log(\lambda_c).
$$
By definition of the pressure gap $\Delta_\mu$, we have $P_{\mu'}(A,\psi)\le P_A(\psi)-\Delta_\mu$, so that
$$
h_{\nu'}(g)+\int\varphi\,d\nu' \;\le\; P_A(\psi)+d_c\log(\lambda_c)-\Delta_\mu \;<\; P_A(\psi),
$$
since $d_c\log(\lambda_c)<\Delta_\mu$.

\smallskip
This proves the Claim. As in the proof of Theorem~\ref{coroequilibrium}, it follows that $P_g(\varphi)=P_A(\psi)$, attained only at $\nu$, and the ergodic decomposition of any equilibrium state of $(g,\varphi)$ together with the affinity of the pressure functional shows $\nu$ is the unique such state.
\end{proof}

\end{document}